\documentclass[12pt]{amsart}
\usepackage{a4wide}
\usepackage[dvipsnames]{xcolor}
\usepackage{mathrsfs}
\usepackage{amsmath, mathtools}
\usepackage{amsthm}
\usepackage{amsfonts}
\usepackage{amssymb}
\usepackage{graphicx}
\usepackage{palatino}
\usepackage{overpic}
\usepackage{hyperref, comment}
\usepackage[noabbrev,capitalize,nameinlink]{cleveref}
\usepackage{svg}
\usepackage[normalem]{ulem}
\usepackage{array}

\usepackage[scaled=0.95]{inconsolata}       
\newtheorem{theorem}{Theorem}[section]
\newtheorem{proposition}{Proposition}[section]
\newtheorem{lemma}{Lemma}[section]
\newtheorem{corollary}{Corollary}[section]
\newtheorem{conjecture}{Conjecture}[section]

\theoremstyle{definition}
\newtheorem{definition}{Definition}[section]
\newtheorem{example}{Example}[section]
\newtheorem{question}{Question}[section]

\theoremstyle{remark}
\newtheorem{remark}{Remark}[section]
\numberwithin{equation}{section}

\crefname{theorem}{Theorem}{Theorems}
\crefname{lemma}{Lemma}{Lemmas}
\crefname{proposition}{Proposition}{Propositions}
\crefname{conjecture}{Conjecture}{Conjectures}
\crefname{definition}{Definition}{Definitions}
\crefname{corollary}{Corollary}{Corollaries}
\crefname{remark}{Remark}{Remarks}
\crefname{question}{Question}{Questions}

\Crefname{theorem}{Theorem}{Theorems}
\Crefname{lemma}{Lemma}{Lemmas}
\Crefname{proposition}{Proposition}{Propositions}
\Crefname{conjecture}{Conjecture}{Conjectures}
\Crefname{definition}{Definition}{Definitions}
\Crefname{corollary}{Corollary}{Corollaries}
\Crefname{remark}{Remark}{Remarks}
\Crefname{example}{Example}{Examples}
\Crefname{question}{Question}{Questions}

\newcommand{\R}{\mathbb{R}}
\newcommand{\C}{\mathbb{C}}
\newcommand{\D}{\mathbb{D}}
\newcommand{\N}{\mathbb{N}}

\newcommand{\norm}[1]{\left\lVert#1\right\rVert}
\newcommand{\ve}{\varepsilon}

\definecolor{darkgreen}{RGB}{0,153,0}
\definecolor{darkred}{RGB}{204,0,0}
\definecolor{darkblue}{RGB}{0,51,204}

\hypersetup{colorlinks,linkcolor={darkblue},citecolor={darkgreen},urlcolor={darkgreen}}

\begin{document}

\title{Regular Lagrangians in Lefschetz fibrations}

\author{Joseph Breen}
\address{University of Alabama, Tuscaloosa, AL 35401}
\email{jjbreen@ua.edu} \urladdr{https://sites.google.com/view/joseph-breen}

\author{Agniva Roy}
\address{Boston College, Chestnut Hill, MA 02467}
\email{agniva.roy@bc.edu} \urladdr{https://sites.google.com/bc.edu/agniva-roy}

\author{Luya Wang}
\address{Institute for Advanced Study, Princeton, NJ 08540}
\email{luyawang@ias.edu} \urladdr{https://www.math.ias.edu/~luyawang/}

\thanks{JB was partially supported by NSF Grant DMS-2038103 and an AMS-Simons Travel Grant. AR was partially supported by an AMS-Simons Travel Grant. LW was supported by NSF Grant DMS-2303437, IAS Giorgio and Elena Petronio Fellow II Fund, and Simons Collaboration - New Structures in Low-dimensional topology.}

\begin{abstract}
    We characterize regularity of Lagrangian submanifolds in Weinstein Lefschetz fibrations, establishing a conjecture of Giroux and Pardon. Our main result is the Weinstein analogue of a closed symplectic Lefschetz pencil result of Auroux, Muñoz, and Presas. As an application, given a Legendrian link in tight $S^3$ and an exact filling which is part of an arboreal skeleton for the $4$-ball, we build a Lefschetz fibration such that the image of the filling and all of its mutations are arcs in the base. 
\end{abstract}

\maketitle

\tableofcontents

\section{Introduction}\label{section:introduction}

We begin with context in \cref{subsec:background} before stating our main theorem in \cref{subsec:maintheorem}. Applications to fillings of Legendrian links are stated in \cref{subsec:applications}.

\subsection{Background}\label{subsec:background}

In \cite{EliashbergGanatraLazarev2018Regular}, Eliashberg, Ganatra, and Lazarev defined \textit{regularity} for properly embedded Lagrangian submanifolds in Weinstein cobordisms.  Informally, a Lagrangian is regular if it is compatible with the surrounding Weinstein handle structure. For instance, regular Lagrangians may be built alongside the ambient Weinstein cobordism with \textit{coupled Weinstein handles}, which are $2n$-dimensional $k$-handles $h^{2n}_k$ containing $n$-dimensional Lagrangian $\ell$-handles $h^{n}_{\ell}$ tangent to the Liouville vector field. Precise definitions and characterizations of regularity are reviewed in \cref{subsec:reg_review}. 

Examples of regular Lagrangians include $0$-sections and cotangent fibers of cotangent bundles (more generally, conormal bundles), co-cores of critical Weinstein handles, and smooth strata of Weinstein skeleta. Moreover, in dimension at least $6$, there is a complete existence $h$-principle \cite{Lazarev2020Hprinciples}. Non-regular Lagrangians exist in cobordisms with negative ends \cite{Murphy2013Exact,Lin2016Caps,DimitroglouRizellGolovko2024NonRegular,DimitroglouRizellGolovko2025NonRegular}, but whether or not there are non-regular Lagrangians in a Weinstein domain is a completely open question. 

In this paper we study regularity in the context of Lefschetz fibrations. As the holomorphic analogue of Morse functions, Lefschetz fibrations bridge complex and symplectic geometry \cite{Donaldson1996SymplecticSubmanifolds,Donaldson1999Lefschetz,Seidel2008Picard,MaydanskiySeidel2010Exotic}, are instrumental in filling classifications in low-dimensional contact topology \cite{Wendl2010Strongly,Plamenevskaya2010Planar,Ozbagci2015Topology}, and have generally led to great insight in Weinstein topology  \cite{Mclean2009Lefschetz,GirouxPardon2017Lefschetz,CasalsMurphy2019Affine}.

Briefly, a \textit{Weinstein Lefschetz fibration} on a Weinstein domain $(W, \lambda,\phi)$ is a smooth submersion $p:W^{2n} \to \D^2$ with critical points locally modeled on the complex quadratic map $(z_1, \dots, z_n) \mapsto \sum_{j=1}^n z_j^2$, such that each regular fiber is a ($2n-2$)-dimensional Weinstein domain. A precise definition is given in \cref{section:preliminaries}; see \cref{def:WLF}.

Given a Weinstein Lefschetz fibration, one obtains a more combinatorial description by first choosing a regular basepoint $\bullet \in \D^2$, letting $W_0^{2n-2} := p^{-1}(\bullet)$, and isotoping the critical values $x_1, \dots, x_k$ to be approximately radially distributed around $\bullet$. Then, one identifies a cyclically ordered tuple of Lagrangian ($n-1$)-spheres in $W_0$ by choosing an approximately straight path $\gamma_j$ from $x_j$ to $\bullet$, and considering the boundary of the corresponding \textit{Lefschetz thimble}, the Lagrangian $n$-disk comprising of points above $\gamma_j$ that parallel transport into the critical point associated to $x_j$. This begets the following definition. 

\begin{definition}\label{def:AWLF}
An \emph{abstract Weinstein Lefschetz fibration} is the data 
\[
\left((W_0^{2n-2}, \lambda_0, \phi_0);\, \mathcal{L} = (L_1, \dots, L_N)\right),
\]
abbreviated $(W_0; \mathcal{L})$, where $(W_0^{2n-2}, \lambda_0, \phi_0)$ is a Weinstein domain and $\mathcal{L}$ is an ordered $N$-tuple of exact parametrized Lagrangian ($n-1$)-spheres (possibly duplicated), called the \emph{vanishing cycles}, embedded in $W_0$. The \emph{total space of $(W_0; \mathcal{L})$}, denoted $|W_0; \mathcal{L}|$, is the $2n$-dimensional Weinstein domain obtained by attaching critical handles to $(W_0\times \D^2, \lambda_0 + \lambda_{\mathrm{st}}, \phi_0 + \phi_{\mathrm{st}})$, where $\lambda_{\mathrm{st}}$ and $\phi_{\mathrm{st}}$ comprise the radial Weinstein structure on the disk (c.f. \cite[Definition 6.3]{GirouxPardon2017Lefschetz}), along Legendrian lifts $\Lambda_j\subseteq W_0\times \partial \D^2$, $j=1, \dots, N$, of $L_j$ positioned near $2\pi j/N \in \partial \D^2$.
\end{definition}

The total space $|W_0; \mathcal{L}|$ of an abstract Weinstein Lefschetz fibration naturally admits a Weinstein Lefschetz fibration $p:|W_0; \mathcal{L}| \to \D^2$. In this paper we will pass fluidly between the abstract and non-abstract formulations, and we refer to the work of Giroux and Pardon \cite[\S 6]{GirouxPardon2017Lefschetz} for more details on how to translate between the two. We will often use $\mathcal{W}$ to stand for the full data of the smoothed Weinstein domain $(W^{\mathrm{sm}},\lambda, \phi)$, and will also use $W_0$ when referring to a general regular fiber (not just over a preferred basepoint).

Importantly, every Weinstein domain arises as the total space of a Weinstein Lefschetz fibration. In dimension $4$, this was first shown by Loi and Piergallini \cite{LoiPiergallini2001Stein} and Akbulut and Özbağci \cite{Akbulut2001Lefschetz} --- see also  Plamanevskaya \cite{Plamenevskaya2004Contact} --- while Giroux and Pardon \cite{GirouxPardon2017Lefschetz} generalized the result to all dimensions. 

\begin{theorem}\cite{GirouxPardon2017Lefschetz}\label{thm:GP17}
Let $(W, \lambda, \phi)$ be a Weinstein domain. There is an abstract Weinstein Lefschetz fibration $(W_0; \mathcal{L})$ whose total space $|W_0; \mathcal{L}|$ is $1$-Weinstein homotopic to $(W,\lambda, \phi)$.   
\end{theorem}

Here we use the language \textit{$1$-Weinstein homotopy} to emphasize that we allow finite instances of singularities that occur in generic $1$-parametric families, i.e. birth-death type singularities. By \textit{$0$-Weinstein homotopy} (or \textit{$0$-homotopy}, if simply discussing smooth Morse functions) we mean a homotopy through non-degenerate structures (or functions).  

\cref{thm:GP17} is the Weinstein counterpart to Donaldson's theorem \cite{Donaldson1999Lefschetz} on existence of symplectic Lefschetz pencil structures for closed integral symplectic manifolds. Both \cite{Donaldson1999Lefschetz} and Giroux and Pardon's \cref{thm:GP17} were proven with Donaldson's approximately holomorphic technology \cite{Donaldson1996SymplecticSubmanifolds,Auroux1997Holomorphic}. More recently, a new proof of \cref{thm:GP17} based on convex hypersurface theory was given in \cite{BreenHondaHuang2023Giroux}. The spirit of the present article is in line with the latter approach.

\subsection{Main results}\label{subsec:maintheorem}

We use advances in convex hypersurface theory \cite{HondaHuang2019Convex,BreenHondaHuang2023Giroux} together with techniques of Johns \cite{Johns2011Lefschetz} and S. Lee \cite{Lee2021Lefschetz} to characterize regularity of Lagrangians in terms of Weinstein Lefschetz fibrations. We establish a conjecture of Giroux and Pardon, first recorded (to the best of the authors' knowledge) by Eliashberg, Ganatra, and Lazarev \cite[\S 2]{EliashbergGanatraLazarev2018Regular}.

\begin{conjecture}[Giroux-Pardon]\label{conj:gp_egl}
If $L\subset W$ is a regular Lagrangian submanifold of a Weinstein domain with $\partial L \neq \emptyset$, there is a Weinstein Lefschetz fibration $p:W \to \D^2$ such that $p(L)$ is an arc with one endpoint on a critical value of $p$ and the other endpoint on $\partial \D^2$.   
\end{conjecture}

\noindent Our (more general) precise statement is recorded below as \cref{thm:main}.

The study of Lagrangian submanifolds via their projections to curves in Lefschetz fibrations --- in particular, the concept of a \textit{matching cycle}, a Lagrangian sphere projecting to an arc connecting two critical values --- was first suggested by Donaldson (see \cite{AurouxMunozPresas2005Pencils,Seidel2015Quartic}), and has been fruitfully exploited in symplectic topology \cite{Seidel2008Picard,MaydanskiySeidel2010Exotic,CasalsMurphy2019Affine,Thompson2025Hopf}.

\begin{definition}\label{def:efficient}
Let $(L; \partial_- L, \partial_+ L)$ be a cobordism. We say that a smooth Morse function $f: L \to \R$ is \emph{efficient} if $f$ has no critical points of index $0$ when $\partial_- L \neq \emptyset$ and exactly one critical point of index $0$ when $\partial_- L = \emptyset$, and likewise $f$ has no critical points of top index when $\partial_+ L \neq \emptyset$ and exactly one critical point of top index when $\partial_+ L = \emptyset$.
\end{definition}

\begin{theorem}\label{thm:main}
Let $(W^{2n}, \lambda, \phi)$ be a Weinstein domain and $(L; \emptyset, \partial L)\subseteq W$ a properly embedded regular Lagrangian submanifold. Let $f: L \to [0,1]$ be an efficient Morse function. Then, there is a Weinstein Lefschetz fibration $p:W \to \D^2$ such that: 
\begin{enumerate}
    \item The total space of $p:W \to \D^2$ is $1$-Weinstein homotopic to $(W, \lambda, \phi)$. 
    \item The image $p(L) \subseteq \D^2$ is a smoothly embedded parametrized arc $\gamma:[0,1]\to p(L)$. 
    \item The function $\gamma^{-1}\circ p\mid_L:L \to [0,1]$ is a Morse function $0$-homotopic to $f:L \to [0,1]$. 
\end{enumerate}
\end{theorem}

\noindent When the conclusion of \cref{thm:main} holds, we will say that $p$ is an {\em admissible Lefschetz fibration} for the tuple $(W,L,f)$, or for $(W,L)$ if the precise function is not important, and that $L$ is an {\em arc-admissible Lagrangian} for $p$.

By taking $L=S^n$ and $f:S^n \to [0,1]$ to be the standard Morse function with two critical points, \cref{thm:main} says that every regular sphere in a Weinstein domain can be witnessed as a matching cycle of a Lefschetz fibration. More generally, \cref{thm:main} is the Weinstein analogue of a similar asymptotic statement for Lagrangian submanifolds in Lefschetz pencils of closed symplectic manifolds due to Auroux, Muñoz, and Presas \cite[Theorem 1.3]{AurouxMunozPresas2005Pencils}. This latter result, also obtained via Donaldson's approximately holomorphic technology, is a version of Donaldson's Lefschetz pencil existence result relative to a prescribed Lagrangian submanifold. Accordingly, \cref{thm:main} is a version of \cref{thm:GP17} relative to a prescribed regular Lagrangian submanifold. 

Giroux and Pardon conjectured that the approximately holomorphic techniques used in their proof of \cref{thm:GP17} could be applied to obtain this corresponding relative statement, in analogy with \cite[Theorem 1.3]{AurouxMunozPresas2005Pencils}; our methods, by contrast, are directly Morse theoretic, instead using the technology of convex hypersurface theory in contact topology. This context is summarized by the tables in \cref{fig:context}.

\begin{figure}[ht]
\centering
\begin{tabular}{|>{\centering\arraybackslash}m{2.7cm}|>{\centering\arraybackslash}m{5.5cm}|>{\centering\arraybackslash}m{6cm}|}
\hline
\multicolumn{3}{|c|}{\textbf{Closed integral symplectic manifolds $(M,\omega)$}} \\
\hline
& Lefschetz pencil existence & Relative (to a Lagrangian) Lefschetz pencil existence \\
\hline
Approx. holo. & \cite{Donaldson1999Lefschetz} & \cite{AurouxMunozPresas2005Pencils} \\
\hline
Morse theory & \textcolor{darkred}{not done yet} & \textcolor{darkred}{not done yet} \\
\hline
\end{tabular}

\vspace{0.5cm}

\begin{tabular}{|>{\centering\arraybackslash}m{2.7cm}|>{\centering\arraybackslash}m{5.5cm}|>{\centering\arraybackslash}m{6cm}|}
\hline
\multicolumn{3}{|c|}{\textbf{Weinstein domains $(W,\lambda, \phi)$}} \\
\hline
& Lefschetz fibration existence & Relative (to a Lagrangian) Lefschetz fibration existence \\
\hline
Approx. holo. & \cite{GirouxPardon2017Lefschetz} & \textcolor{darkred}{not done yet;} proposed to solve \cref{conj:gp_egl}\\
\hline
Morse theory & \cite{BreenHondaHuang2023Giroux} & \cref{thm:main} (\cref{conj:gp_egl}) \\
\hline
\end{tabular}
\caption{Comparison of results and techniques for closed symplectic manifolds and Weinstein domains.}
\label{fig:context}
\end{figure}

The core of the proof of \cref{thm:main} is the case of the $0$-section of a cotangent bundle. This is where we use and further develop the approach of S. Lee \cite{Lee2021Lefschetz}, generalizing low-dimensional work of Johns \cite{Johns2011Lefschetz}, on constructing Lefschetz fibrations from Weinstein handle decompositions. 

\begin{remark}
As we prepared to post our work publicly, Giroux uploaded the preprint \cite{giroux2025morsefunctionslefschetzfibrations} which considers explicit constructions of Lefschetz fibrations on cotangent bundles relative to a chosen Morse function on the $0$-section. This recovers the case of the $0$-section in the cotangent bundle of \cref{thm:main}, reformulating and refining work of Johns and Lee with more direct techniques.    
\end{remark}

\subsection{Applications}\label{subsec:applications}
We apply the above results to visualize fillings of Legendrian links $\Lambda \subset (S^3, \xi_{\mathrm{st}})$ in the unique (up to deformation equivalence) Weinstein filling $(B^4, \lambda_{\mathrm{st}}, \phi_{\mathrm{st}})$. A properly embedded Lagrangian $L \subset B^4$ is called an {\em exact filling} of $\Lambda$ if $\lambda_{\mathrm{st}}|_L = 0$ and $\partial L = \Lambda$. We will consider exact fillings $L$ of $\Lambda$ with an associated set $\mathcal{D}$ of Lagrangian disks in the interior of $B^4$, with boundary on $L$, such that $L \cup \mathcal{D}$ forms a {\em closed arboreal Lagrangian} skeleton (abbreviated CAL-skeleton) for the Weinstein pair $(B^4, \Lambda)$. We define these notions in \cref{sec:mutation}; for now, we note that regular Lagrangians are exact, and smooth loci of CAL-skeleta are regular. Given this data, one can perform {\em Lagrangian mutation} of $L$ along one of the disks to produce a new exact filling.

Using the argument for \cref{thm:main}, we build a Lefschetz fibration $p_L$ on $(B^4, \lambda_{\mathrm{st}}, \phi_{\mathrm{st}})$ such that $p_L$ is admissible for  $(B^4, L\cup \mathcal{D})$, and then show the following.

\begin{theorem}\label{thm: mutation}
    Let $\Lambda, L, \mathcal{D},$ and $p_L: (B^4, \lambda_{\mathrm{st}}, \phi_{\mathrm{st}})\to \D^2$ be as above. All fillings of $\Lambda$ obtained by Lagrangian mutations of $L$ along disks in $\mathcal{D}$ are arc-admissible for $p_L$.
\end{theorem}

It is natural to ask which exact fillable Legendrian links admit such fillings with associated CAL-skeleta, or more generally which exact fillable links admit regular fillings (as asked in \cite{Breen2024Regularly}). The question of whether an exact filling of a Legendrian link is part of a CAL-skeleton also has consequences to associated cluster structures. In this case, work of Casals and Li \cite{casals-li24} ensures that microlocal merodromies along positive relative cycles on the filling extend to global regular functions on the sheaf moduli of the boundary Legendrian. In particular, this allows for the symplectic construction of cluster structures on these sheaf moduli (i.e., the main results of \cite{CasalsWeng} and \cite{CGGLSS}) without relying on combinatorics. In analogy to Lazarev's $h$-principle result in higher dimensions \cite{Lazarev2020Hprinciples}, we formulate the following question.
\begin{question}
    Is every exact Lagrangian filling of a Legendrian link in tight $S^3$ formally Lagrangian isotopic to a regular filling, of a formally Legendrian isotopic link, with an associated CAL-skeleton?
\end{question}
\noindent Note that without the assumption of exactness, this question admits a negative answer --- see the discussion after \cite[Theorem 1.1]{Lazarev2020Hprinciples}.

A weaker question is the following: 

\begin{question}\label{q:new_q}
Given any exact fillable Legendrian link $\Lambda$, does there exists a Legendrian link $\Lambda'$, smoothly isotopic to $\Lambda$, such that $\Lambda'$ has a regular filling $L$ with an associated CAL-skeleton $L\cup\mathcal{D}$ for $(B^4, \Lambda')$?
\end{question}

Work of Casals \cite[Theorem 1.1]{casals_skeleta} establishes this affirmatively for (the unique) max-tb Legendrian representatives of links of single algebraic singularities; such links are iterated cables of the unknot. Combining work of Boileau and Orevkov \cite{boileau-orevkov} and Rudolph \cite{rudolph2004algebraic}, every exact fillable topological link type is algebraic, in that it bounds a smooth algebraic curve in $\C^2$. Thus, one approach to \cref{q:new_q} is to extend Casals' work on links of single singularities to the more general class of algebraic links as described in Rudolph \cite{rudolph2004algebraic}.

We also give an independent construction of CAL-skeleta associated to decomposable fillings of certain links by adapting an argument of \cite{ConwayEtnyreTosun2021Disks}: 

\begin{proposition}\label{prop:decomposable-skel}
     If $\Lambda_\beta$ is the rainbow closure of a positive braid $\beta$ (see \cref{fig:rainbow}), the Weinstein pair $(B^4, \Lambda_\beta)$, admits CAL-skeleta associated to decomposable fillings of $\Lambda_\beta$.
\end{proposition}

\begin{figure}[t]
	\begin{overpic}[scale=.3]{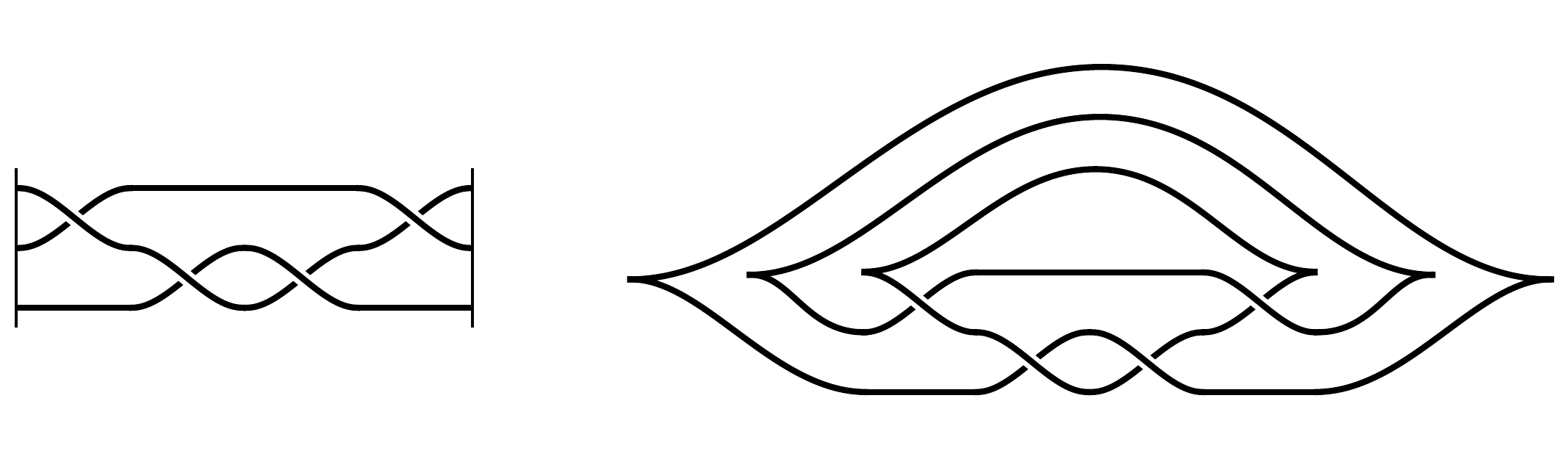}   
	\end{overpic}
    \vskip-0.1cm
	\caption{A positive braid $\beta$ and its rainbow closure $\Lambda_{\beta}$.}
	\label{fig:rainbow}
\end{figure}

For Legendrian links of this form, exact Lagrangian fillings have been constructed using decomposable cobordisms \cite{ekholm2012exactcobordisms,pan2016fillings}, Legendrian weaves \cite{TreumannZaslow,casals2022legendrian,casals_gao24,cggs24}, and conjugate fillings \cite{STWZ,casals-li22}. It is shown \cite{casals-li22,hughes_a-type} that all these constructions are Hamiltonian isotopic. Further, it is known \cite{casals2024microlocal} that the sheaf moduli $\mathcal{M}_1(\Lambda_\beta)$ admits a cluster structure, and every cluster seed of the cluster structure has a corresponding filling \cite{casals_gao24}. Specifically, Casals and Gao use an initial filling $L_{\mathrm{init}}$ constructed using conjugate fillings, and show that all the other fillings can be obtained by the process of  Lagrangian mutations. Conjecturally, all fillings of $\Lambda_\beta$ are obtained in this way. We can summarize the above discussion in the following:

\begin{corollary}\label{cor:CG-fillings}
    Given $\Lambda_\beta$, the rainbow closure of a positive braid $\beta$, there exists a Lefschetz fibration $p_\beta: (B^4, \lambda_{\mathrm{st}}, \phi_{\mathrm{st}}) \to \D^2$ such that all fillings constructed by Casals and Gao corresponding to all the seeds of the cluster variety $X_\beta$ are arc-admissible for $p_\beta$.
\end{corollary}

\subsection{Organization}

In \cref{section:preliminaries} we collect background information on coupled Weinstein handles, regular Lagrangians, and open book decompositions. In \cref{section:coupled_lefschetz_handles} we review and reformulate the work of Johns \cite{Johns2011Lefschetz} and S. Lee \cite{Lee2021Lefschetz} in a framework we call \textit{coupled Lefschetz handles}, and in \cref{sec:proof} we rephrase and prove \cref{thm:main}. Finally, in \cref{sec:mutation} we apply the above results to visualizing exact Lagrangian fillings of Legendrian links in $(S^3, \xi_{\mathrm{st}})$. We build a particular Lefschetz fibration compatible, in the sense of \cref{thm:main}, with a filling $L$ and its associated CAL-skeleton. We then show that fillings obtained by Lagrangian mutations are also compatible with the Lefschetz fibration, proving \cref{thm: mutation}. We finish by constructing fillings with associated CAL-skeleta, and proving \cref{prop:decomposable-skel} and \cref{cor:CG-fillings}.

\subsection{Acknowledgments}

The authors thank James Hughes and Oleg Lazarev for interest in our work and comments on preliminary draft. We thank Eduardo Fernandez and Kyle Hayden for helpful observations concerning an earlier version of this work related to \cref{q:new_q}. We additionally thank Orsola Capovilla-Searle, Austin Christian, John Etnyre, and Ko Honda for helpful discussions. 

\section{Preliminaries}\label{section:preliminaries}

In this section we provide specific background material on coupled Weinstein handles, regularity of Lagrangians, and open book decompositions. Beyond our exposition, none of the content in this section is original. In general, standard references for contact topology and Weinstein topology include \cite{Geiges2008Introduction,CieliebakEliashberg2012FromStein,Eliashberg2018Weinstein}.

\subsection{Coupled Weinstein handles}\label{subsec:coupled}

The model for a Weinstein $k$-handle in dimension $2n$, $0\leq k\leq n$, is the data $(h_{k}^{2n}, \omega, X_{\lambda}, \phi)$, where: 
\begin{align*}
    h^{2n}_k &= \D^k_{(y_1, \dots, y_k)} \times \D^{2n-k}_{(y_{k+1}, \dots, y_n, x_1, \dots, x_n)}, \\
    \omega &= \sum_{j=1}^n dx_{j} \wedge dy_j, \\
    X_{\lambda} &= \sum_{j=1}^k \left(2x_j\, \partial_{x_j}-y_j\, \partial_{y_j}\right) \, + \, \sum_{j=k+1}^n \left(\frac{1}{2}x_j\, \partial_{x_j} + \frac{1}{2}y_j\, \partial_{y_j}\right), \\
    \phi &= \sum_{j=1}^k \left(x_j^2-\frac{1}{2}y_j^2\right) \, + \, \sum_{j=k+1}^n \left(\frac{1}{4}x_j^2+\frac{1}{4}y_j^2\right).
\end{align*}
In the latter two expressions, the first (resp. second) sum is $0$ when $k=0$ (resp. $k=n$). Note that 
\[
\lambda:= \iota_{X_{\lambda}}\omega = \sum_{j=1}^k \left(2x_j\, dy_j + y_j\, dx_j\right) \, + \, \sum_{j=k+1}^n \left(\frac{1}{2}x_j\, dy_j - \frac{1}{2}y_j\, dx_j\right)
\]
is a primitive for $\omega$, hence $\lambda$ determines both $\omega = d\lambda$ and its Liouville vector field $X_{\lambda}$. It is therefore equivalent to refer to the data $(h_k^{2n}, \lambda, \phi)$. In practice, we will typically suppress all of the data and simply refer to a Weinstein $k$-handle by the notation $h_{k}^{2n}$.  

\begin{remark}
The core 
\[
\D^k_{(y_1, \dots, y_k)} \times \{0\}_{(y_{k+1}, \dots, y_n,x_1,\dots,x_n)} = \{y_{k+1} = \cdots = y_n =  x_1 = \cdots x_n = 0\}
\]
and the co-core 
\[
\{0\}_{(y_1, \dots, y_k)} \times \D^{2n-k}_{(y_{k+1}, \dots, y_n, x_1, \dots, x_n)} = \{y_1=\cdots = y_k = 0\}
\]
are isotropic and co-isotropic disks, respectively, which are both Lagrangian when $k=n$. In this case, the attaching and belt spheres are Legendrian spheres in their respective boundary components. 
\end{remark}

The notion of a coupled Weinstein handle, due to \cite{EliashbergGanatraLazarev2018Regular}, identifies a model Lagrangian $\ell$-handle inside the model Weinstein $k$-handle. 

\begin{definition}
A $2n$-dimensional \emph{coupled Weinstein $(k, \ell)$-handle} for $0\leq \ell \leq k \leq n$ is a pair $(h_{k}^{2n}, h^n_{\ell})$ where $h_k^{2n}$ is a $2n$-dimensional Weinstein $k$-handle as above and $h^n_{\ell} \subseteq h_k^{2n}$ is an $n$-dimensional Lagrangian $\ell$-handle properly embedded in $h_k^{2n}$ as 
\[
h^n_{\ell} = \D^{\ell}_{(y_1, \dots, y_\ell)} \times \D^{n-\ell}_{(x_{\ell+1},\dots,x_n)} = \{y_{\ell+1} = \cdots =y_n = x_1 = \cdots =x_{\ell}= 0\}.
\]
When $\ell=k$ we unambiguously refer to the pair as a \emph{coupled Weinstein $k$-handle}. For linguistic cohesion, we call the pair $(h_k^{2n}, \emptyset)$ consisting of just a Weinstein $k$-handle a \emph{trivial coupled Weinstein $k$-handle}. 
\end{definition}

Observe that the index-$k$ Liouville vector field $X_{\lambda}$ on $h_{k}^{2n}$ restricts to an index-$\ell$ vector field on $h_{\ell}^n$. There is not a unique such choice of Lagrangian handle inside a Weinstein handle, so the specification of $h_{\ell}^n$ in the model is additional data. See \cref{fig:coupled_handle} for the models of various $4$-dimensional coupled Weinstein $(2,\ell)$-handles. 

\begin{example}
    If $C$ and $K$ denote the core and co-core, respectively, of a critical Weinstein $n$-handle $h_n^{2n}$, then $(h_n^{2n}, C)$ is a coupled $(n, n)$-handle and $(h_n^{2n}, K)$ is a coupled $(n, 0)$-handle. 
\end{example}

\begin{figure}[ht]
	\begin{overpic}[scale=.42]{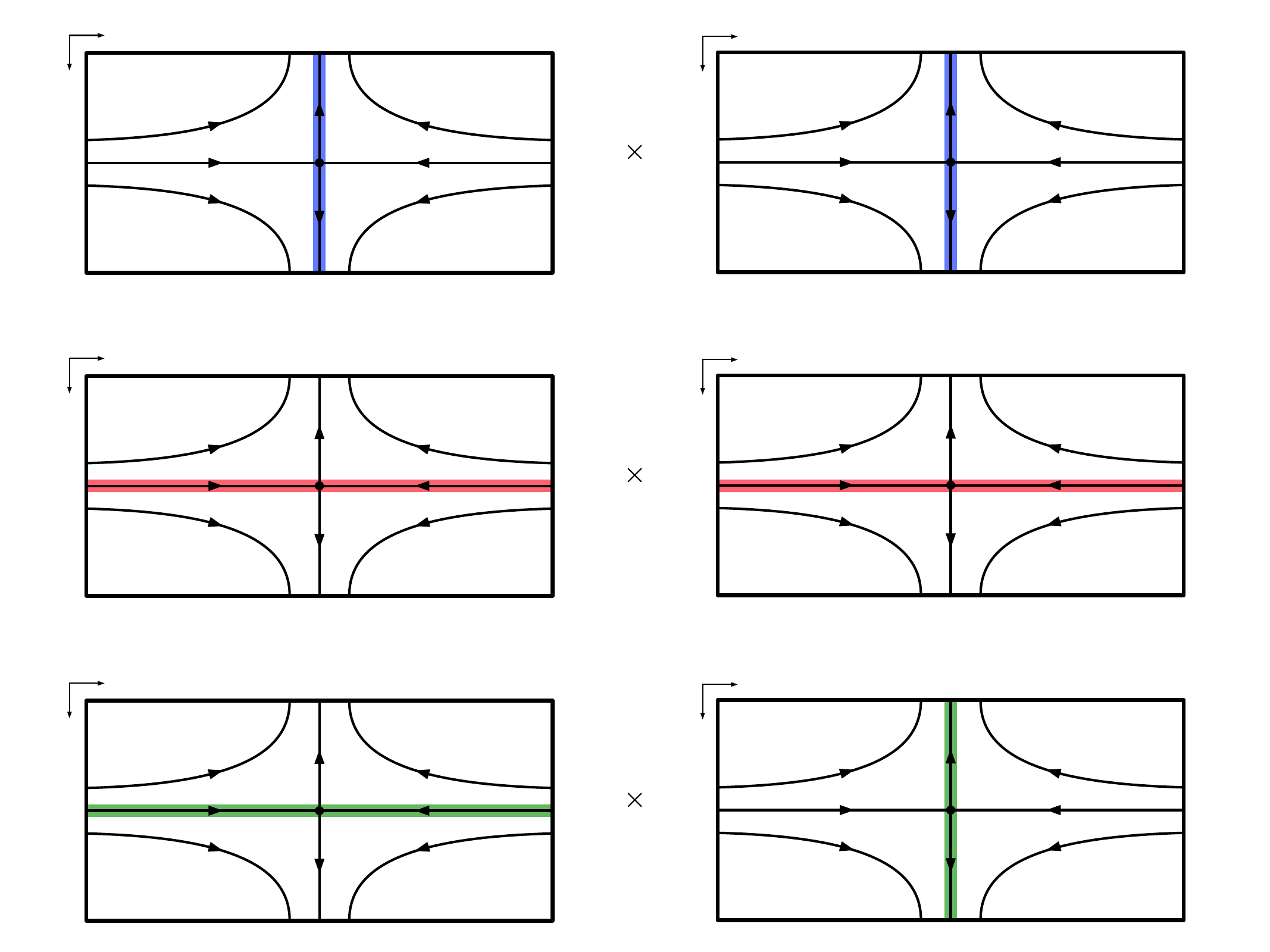}   
        \put(4.5,17){\tiny $x_1$}
        \put(8.5,21){\tiny $y_1$}
        \put(54.25,17){\tiny $x_2$}
        \put(58.5,21){\tiny $y_2$}
        
        \put(4.5,42.5){\tiny $x_1$}
        \put(8.5,46.5){\tiny $y_1$}
        \put(54.25,42.5){\tiny $x_2$}
        \put(58.5,46.5){\tiny $y_2$}

        \put(4.5,68){\tiny $x_1$}
        \put(8.5,72){\tiny $y_1$}
        \put(54.25,68){\tiny $x_2$}
        \put(58.5,72){\tiny $y_2$}

        \put(45,58){\tiny \textcolor{darkblue}{$(2,0)$-handle}}
        \put(45,33){\tiny \textcolor{darkred}{$(2,2)$-handle}}
        \put(45,7){\tiny \textcolor{darkgreen}{$(2,1)$-handle}}
        
	\end{overpic}
	\caption{Each row depicts a $4$-dimensional Weinstein $2$-handle $h_2^4 = \D^2 \times \D^2$ with Liouville vector field $X_{\lambda}$ in black. In each row, the highlighted portions specify a $2$-dimensional Lagrangian $\ell$-handle $h_{\ell}^2$ for $\ell=0$ in blue, $\ell=2$ in red, and $\ell=1$ in green. Note that the blue and red Lagrangian handles correspond to the co-core and core of the Weinstein handle, respectively.}
	\label{fig:coupled_handle}
\end{figure}
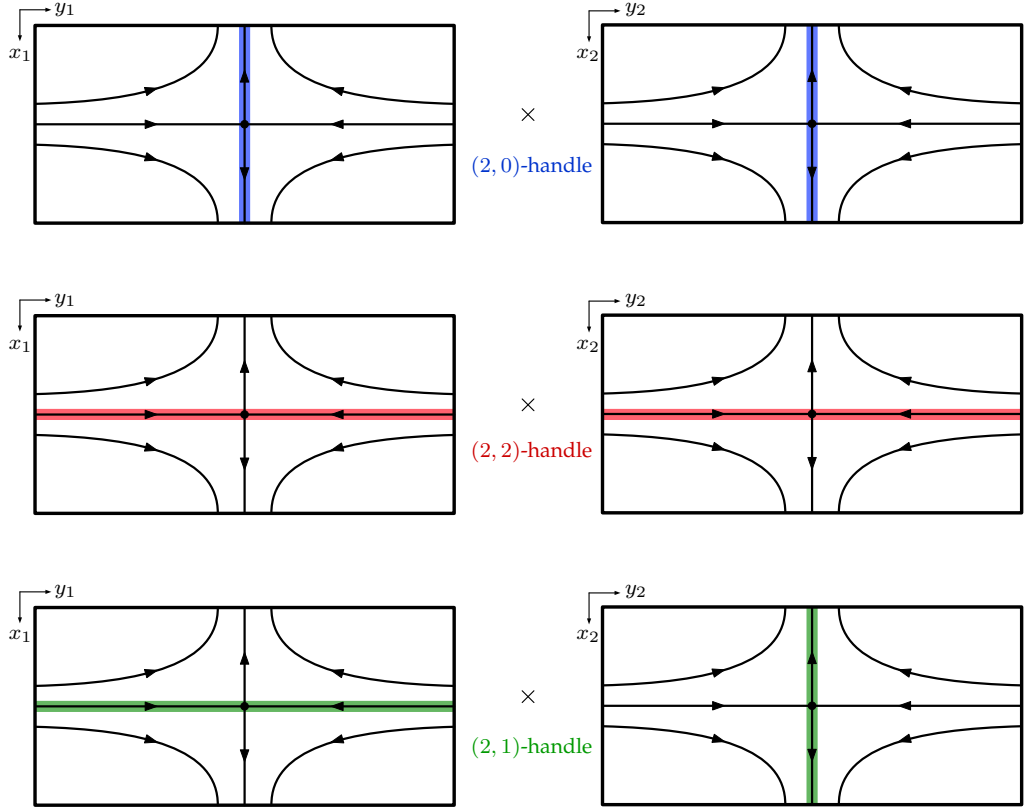

Coupled handle attachments are performed so that one simultaneously builds a Lagrangian submanifold via the Lagrangian handles inside the Weinstein cobordism built by the ambient Weinstein handles. To this end, we adopt the following language. 

\begin{definition}\label{def:coupled_handle_decomposition}
Let $L^n \subseteq (W^{2n}, \lambda, \phi)$ be a Lagrangian submanifold inside a Weinstein cobordism. We say that the pair $(W,L)$ admits a \emph{coupled Weinstein handle decomposition} if there is a sequence of coupled Weinstein handle attachments such that the ambient Weinstein handles give a handle decomposition of $W$ and the Lagrangian handles give a handle decomposition of $L$. 
\end{definition}

\begin{example}\label{example:coupled_handles}
Let $W= \D^*S^2$ be the disk cotangent bundle of $S^2$ and let $L\subseteq W$ be the $0$-section. We describe two different coupled Weinstein handle decompositions of $(W,L)$.
\begin{enumerate}
\item First, consider $(h^4_0, D) \cup (h^4_2, C)$ where $D=h_0^2$ is the standard Lagrangian disk filling of the standard Legendrian unknot $\partial D \subseteq (S^3, \xi_{\mathrm{st}}) = \partial h^4_0$, $C= h_2^2$ is the core of the critical handle $h^4_2$, and $h_2^4$ is attached along $\partial D$. Then $h^4_0 \cup h^4_2 = \D^*S^2$ and $D \cup C$ is the $0$-section. 

\item Next consider
\[
(h^4_0, \emptyset) \cup (h^4_1, \emptyset) \cup (h_2^4, K) \cup (\tilde{h}^4_2, C)
\]
where $h_2^4$ is a Weinstein $2$-handle that cancels the Weinstein $1$-handle $h^4_1$, $K=h_0^2$ is the co-core of $h^4_2$, and $C= h_2^2$ is the core of $\tilde{h}^4_2$. We attach $\tilde{h}^4_2$ along $\partial K$. The resulting Weinstein domain is $1$-Weinstein homotopic to $\D^*S^2$ and the resulting Lagrangian $K\cup C$ is identified under this homotopy with the $0$-section.
\end{enumerate}
\end{example}

\begin{example}
Elementary Lagrangian cobordisms inside a symplectization can be witnessed via a sequence of coupled Weinstein handle attachments. Fix $0\leq k \leq n-1$ and consider a coupled Weinstein $k$-handle $(h_k^{2n}, h_k^n)$ followed by a trivial coupled Weinstein ($k+1$)-handle $(h_{k+1}^{2n}, \emptyset)$, such that the pair $h_k^{2n} \cup h_{k+1}^{2n}$ is canceling at the level of Weinstein handles. Up to $1$-Weinstein homotopy, the result is a trivial Weinstein cobordism (homotopic to a symplectization) which contains an elementary index $k$ exact Lagrangian cobordism. Adopting the language of \cite{GanatraPardonShende2024Sectorial}, we call such a sequence of handle attachments an \emph{exact embedded Lagrangian $k$-handle}; see also \cite{DimitroglouRizellGolovko2012Legendrian,BourgeoisSabloffTraynor2015Generating}. (Note that this construction is more or less part of the proof that decomposable Lagrangians are regular, due to \cite{ConwayEtnyreTosun2021Disks}.)
\end{example}

\subsection{Regular Lagrangians}\label{subsec:reg_review}

We recall the notion of regularity for Lagrangian submanifolds as described in \cite[\S 2]{EliashbergGanatraLazarev2018Regular}. Both the definition of regularity and various equivalent notions are summarized with the following proposition. 

\begin{proposition}[Characterizations of regularity]\label{prop:regularity_characterization}
Let $(W^{2n}, \lambda, \phi)$ be a Weinstein domain and $(L, \partial L) \subseteq (W, \partial W)$ a properly embedded Lagrangian submanifold with (possibly empty) Legendrian boundary. The following are equivalent. 
\begin{enumerate}
    \item The Lagrangian $L$ is \emph{regular}. 
    \item There is a $1$-Weinstein homotopy from $(\lambda, \phi)$ to $(\lambda', \phi')$ for which $L$ remains Lagrangian and such that the Liouville vector field $X_{\lambda'}$ is everywhere tangent to $L$, i.e. $\lambda'\mid_L \equiv 0$. For this reason we also say that the Weinstein structure $(\lambda', \phi')$ is \emph{tangent to $L$}. 
    \item There is a $1$-Weinstein homotopy from $(\lambda, \phi)$ to $(\lambda', \phi')$ for which $L$ remains Lagrangian such that $(W, \lambda', \phi')$ admits a decomposition
    \[
    W \cong \D^*L \cup W_L
    \]
    where, for some regular value $c$ of $\phi'$, $\D^*L \cong \{\phi' \leq c\}$ is the standard Weinstein structure on $\D^*L$, and $W_L \cong \{\phi' \geq c\}$, the \emph{complementary cobordism to $L$}, is a Weinstein cobordism.

    \item The pair $(W, L)$ admits a coupled Weinstein handle decomposition as in \cref{def:coupled_handle_decomposition}.
\end{enumerate}
\end{proposition}

\begin{remark}
Informally, (3) says that $L\subseteq W$ is regular if, up to Weinstein homotopy, there is a handle presentation of $W$ given by starting with $\D^*L$ and successively attaching Weinstein handles. Note that if $L$ has boundary, the Weinstein structure on $\D^*L$ must be chosen so that the Liouville vector field points out of the vertical boundary $\D^*L\mid_{\partial L}$. Moreover, we can require the attaching locus of the additional handles to avoid $\partial L$. 
\end{remark}

Regarding the ``standard" Weinstein structure on the cotangent bundle, we appeal to the following lemma, which is more or less \cite[Lemma 12.8]{CieliebakEliashberg2012FromStein}; we adopt the specific statement from \cite[Lemma 3.9]{Breen2024Regularly}, where there is also a detailed proof. Given a manifold $L$, we denote by $\lambda_{\mathrm{st}} = \vec{p}\, d\vec{q}$ the canonical Liouville form on the cotangent bundle, where $\vec{q}$ is any local coordinate system on $L$ and $\vec{p}$ is the corresponding dual coordinate system.

\begin{lemma}[The standard Weinstein structure on the cotangent bundle]\label{lemma:cot_bund_str}
Let $L$ be a compact $n$-dimensional cobordism equipped with a Riemannian metric. Let $X\in \mathfrak{X}(L)$ be any vector field. There is a Liouville homotopy from the canonical Liouville form $\lambda_{\mathrm{st}} = \vec{p}\, d\vec{q}$ to a Liouville form $\lambda$, compactly supported outside of a neighborhood of the horizontal boundary $\partial_{\mathrm{hor}} \D^*L = \mathbb{S}^*L$, with the following properties. 
\begin{enumerate}
    \item The Liouville vector field $X_{\lambda}$ coincides with $X$ along the $0$-section.\label{cot_lvf_1} 
    
    \item If $X$ is gradient-like for a function $\phi:L \to \R$, and if the eigenvalues of $X$ at zeroes have real part $< 1$, then $X_{\lambda}$ is gradient-like for a function $\hat{\phi}:\D^*L \to \R$ that agrees with $\ve \phi(\vec{q}) + \norm{\vec{p}}^2$ near $\{\vec{p}=0\}$ for some sufficiently small $\ve > 0$, has no critical points away from $\{\vec{p}=0\}$, and agrees with $\norm{\vec{p}}^2$ near $\partial_{\mathrm{hor}} \D^*L$.\label{cot_lvf_2}
\end{enumerate} 
In particular, if $L$ is a cobordism  and $\phi:L \to \R$ is Morse, then there is a gradient-like vector field $X$ for $\phi$ such that $(\D^*L, \lambda, \hat{\phi})$ is Weinstein. Moreover, any two such choices of Weinstein structures are Weinstein homotopic, and if the two structures agree near $\partial_{\pm} L$ then the homotopy may be taken to be constant near $\partial_{\pm} L$.   
\end{lemma}

In the case that $L$ is a disk, there is another nice characterization of regularity. 

\begin{proposition}\cite[Proposition 2.3]{EliashbergGanatraLazarev2018Regular}
Let $D\subseteq (W^{2n}, \lambda, \phi)$ be a properly embedded Lagrangian disk with Legendrian boundary $\partial D \subseteq \partial W$. Then $D$ is regular if and only if, up to $1$-Weinstein homotopy, there is a Weinstein handle decomposition of $W$ such that $D$ is the co-core of a critical handle.   
\end{proposition}

\subsection{Lefschetz fibrations and open book decompositions}

First, we carefully state the data of a Weinstein Lefschetz fibration.

\begin{definition}\label{def:WLF}
A \emph{Weinstein Lefschetz fibration} is the data $(p:W^{2n} \to \D^2, \lambda,\phi)$ satisfying the following properties: 
\begin{enumerate}
    \item The manifold $W^{2n}$ is a compact domain with corners and a smoothing $W^{\mathrm{sm}}$ such that the \emph{total space} $(W^{\mathrm{sm}}, \lambda, \phi)$ is a Weinstein domain, where $\lambda$ is a Liouville form and $\phi$ is a Lyapunov function for the corresponding Liouville vector field with $\partial W^{\mathrm{sm}}$ as a level set.
    \item The map $p:W \to \D^2$ is a smooth fibration except at finitely many critical points in the interior of $W$ with distinct critical values, around which there are local holomorphic coordinates such that 
\[
p(z) = p(z_0) + \sum_{j=1}^n z_j^2 \quad \text{ and } \quad \lambda = i\sum_{j=1}^n (z_j\, d\overline{z}_j -\overline{z}_j \,dz_j).
\]

    \item The boundary $\partial W$ decomposes as  
    \[
    \left[\partial_{\mathrm{vert}} W := p^{-1}(\partial \D^2)\right] \,\cup\, [\partial_{\mathrm{hor}} W: = \bigcup_{x\in \D^2} \partial ( p^{-1}(x))]
    \]
    meeting at a codimension-$2$ corner, where $p|_{\partial_{\mathrm{vert}} W}$ and $p|_{\partial_{\mathrm{hor}} W}$ are fiber bundles.
    \item On each regular fiber, $\lambda$ induces a Weinstein structure such that the contact form $\lambda\mid_{\partial (p^{-1}(x))}$ is independent of $x\in \D^2$, and $d\lambda$ is nondegenerate on $\ker dp(z)$ for all $z\in W$ (not just regular points). 
\end{enumerate}
\end{definition}

An \emph{open book decomposition} of a closed co-oriented contact manifold $(M^{2n-1}, \xi)$ is a pair $(B, \pi)$, where 
\begin{enumerate}
    \item $B^{2n-3}\subset M$ is a codimension-$2$ contact submanifold called the \emph{binding},
    \item $\pi: M \setminus B \to S^1$ is a fibration that agrees with an angular coordinate $\theta$ in a tubular neighborhood $B\times \D^2_{(r, \theta)}$ of $B$, and 
    \item there is a Reeb vector field $R_{\alpha}$ for $\xi$ everywhere transverse to each \emph{page} $\pi^{-1}(\theta)$. 
\end{enumerate}
If $(\pi^{-1}(\theta), \alpha\vert_{\pi^{-1}(\theta)})$ admits the structure of a completed Weinstein domain for every $\theta\in S^1$, we say that the open book decomposition is \emph{strongly Weinstein}. The monodromy of the fibration $\pi$ is an exact symplectomorphism, and thus gives rise to an \textit{abstract open book}, i.e. a pair $(W_0, \psi)$ where $W_0^{2n-2}$ is a Weinstein domain and $\psi:W_0 \to W_0$ is an exact symplectomorphism which is the identity near $\partial W_0$. A \emph{positive stabilization} of an abstract open book is the passage from $(W_0,\psi)$ to $(W_0\cup h_{n-1}^{2n-2}, \tau \circ\psi)$ where $h_{n-1}^{2n-2}$ is a critical handle attached along any Legendrian sphere $\Lambda \subset \partial W$ such that $\Lambda$ bounds a (regular) Lagrangian disk $D\subset W_0$, and $\tau$ is the symplectic Dehn twist performed along the exact Lagrangian sphere $D \cup C$, where $C$ is the core of $h_{n-1}^{2n-2}$. We call a sequence of stabilizations a \textit{multistabilization}.  Positively stabilizing an open book preserves the contactomorphism type of the contact manifold, and in fact: 

\begin{theorem}\cite{Giroux2002ICM,BreenHondaHuang2023Giroux}\label{thm:GC}
Any two strongly Weinstein open book decompositions of a contact manifold admit a common positive multistabilization.     
\end{theorem}

The boundary of a Weinstein Lefschetz fibration inherits a strongly Weinstein open book decomposition described precisely by an abstract presentation. Namely, for an abstract Weinstein Lefschetz fibration $(W_0; \mathcal{L})$, the fibration on the total space $|W_0; \mathcal{L}|$ induces the abstract open book $(W_0, \tau_{\mathcal{L}})$ of the boundary contact manifold $\partial |W_0; \mathcal{L}|$, where $\tau_{\mathcal{L}} = \tau_{L_N} \circ \cdots \circ \tau_{L_1}$ is the composition of Dehn twists along each vanishing cycle. Moreover, a positive stabilization of the boundary contact manifold is induced by a \textit{Lefschetz stabilization}: the attachment of a critical handle $h^{2n-2}_{n-1}$ to the regular fiber, together with a new vanishing cycle $L$, included at the front of the tuple of vanishing cycles $\mathcal{L}$, 
such that $L\cap h^{2n-2}_{n-1}$ is the core of the new handle. 

The other main result we will use is an existence theorem for strongly Weinstein supporting open book decompositions whose pages support a prescribed Legendrian submanifold. That this is possible in dimension $3$ was largely implicit in Giroux's description of contact cell decompositions in \cite{Giroux2002ICM}, though was first written down by Akbulut and Özbağci \cite{Akbulut2001Lefschetz} and Plamanevskaya \cite{Plamenevskaya2004Contact} --- see also the algorithm of Avdek \cite{avdek2013surgery}.

\begin{theorem}\cite{HondaHuang2019Convex}\label{thm:obd_exist_leg_page}
Let $\Lambda^{n-1} \subseteq (M^{2n-1}, \xi)$ be a closed Legendrian submanifold. Then $(M, \xi)$ admits a strongly Weinstein supporting open book decomposition $(B,\pi)$ with Weinstein page $W_0^{2n-2}$ such that $\Lambda \subseteq W_0$ embeds in a page as a regular Lagrangian.      
\end{theorem}

\section{Coupled Lefschetz handles}\label{section:coupled_lefschetz_handles}

A coupled Weinstein $\ell$-handle $(h_{\ell}^{2n}, h_{\ell}^n)$ modifies both the topology of the ambient Weinstein domain and the embedded Lagrangian of interest by attaching a handle of index $\ell$. This section describes how to modify the Weinstein structure on such a coupled handle to make it compatible with a Lefschetz fibration. Here we use the work of S. Lee \cite{Lee2021Lefschetz}, generalizing Johns \cite{Johns2011Lefschetz} in dimension $4$, in decomposing a given Weinstein handle into three parts to make it Lefschetz compatible.

\subsection{Coupled Lefschetz handlebodies}

We begin with the following observation regarding subcritical coupled Weinstein handles. 

\begin{lemma}\label{lemma:coupled_weinstein_lefschetz_handle_htpy}
Fix $0\leq \ell \leq n-1$. Let $(h_{\ell}^{2n}, h_\ell^n)$ be a coupled Weinstein $\ell$-handle. Then $(h_\ell^{2n}, h_\ell^n)$ is $1$-Weinstein homotopic to a sequence of coupled Weinstein handle attachments of the form 
\[
(h_\ell^{2n}, \emptyset) \cup (h_{n-1}^{2n}, \emptyset) \cup (h_{n}^{2n}, h_\ell^n)
\]
where $h_{n-1}^{2n}, h_{n}^{2n}$ are in canceling position. Moreover, we may require the $1$-Weinstein homotopy to be compactly supported in the interior of $h_\ell^{2n}$ and to be tangent to the Lagrangian handle $h_\ell^n$ throughout the homotopy.    
\end{lemma}

\begin{proof}
A coupled Weinstein $\ell$-handle $(h_\ell^{2n}, h_\ell^n)$ admits a decomposition
\begin{equation}\label{eq:decomp1}
    \left(h_\ell^{2n}, h_\ell^n\right) = \left(h_\ell^{2\ell} \times h_0^{2(n-\ell)},\, h_\ell^\ell \times h_0^{n-\ell}\right) = \left(h_\ell^{2\ell},h_\ell^\ell\right) \times \left(h_0^{2(n-\ell)}, h_0^{n-\ell}\right).
\end{equation}
See the first row of \cref{fig:lefschetz_handle} for the case $n=2$ and $\ell=1$. 

\begin{figure}[ht]
	\begin{overpic}[scale=.34]{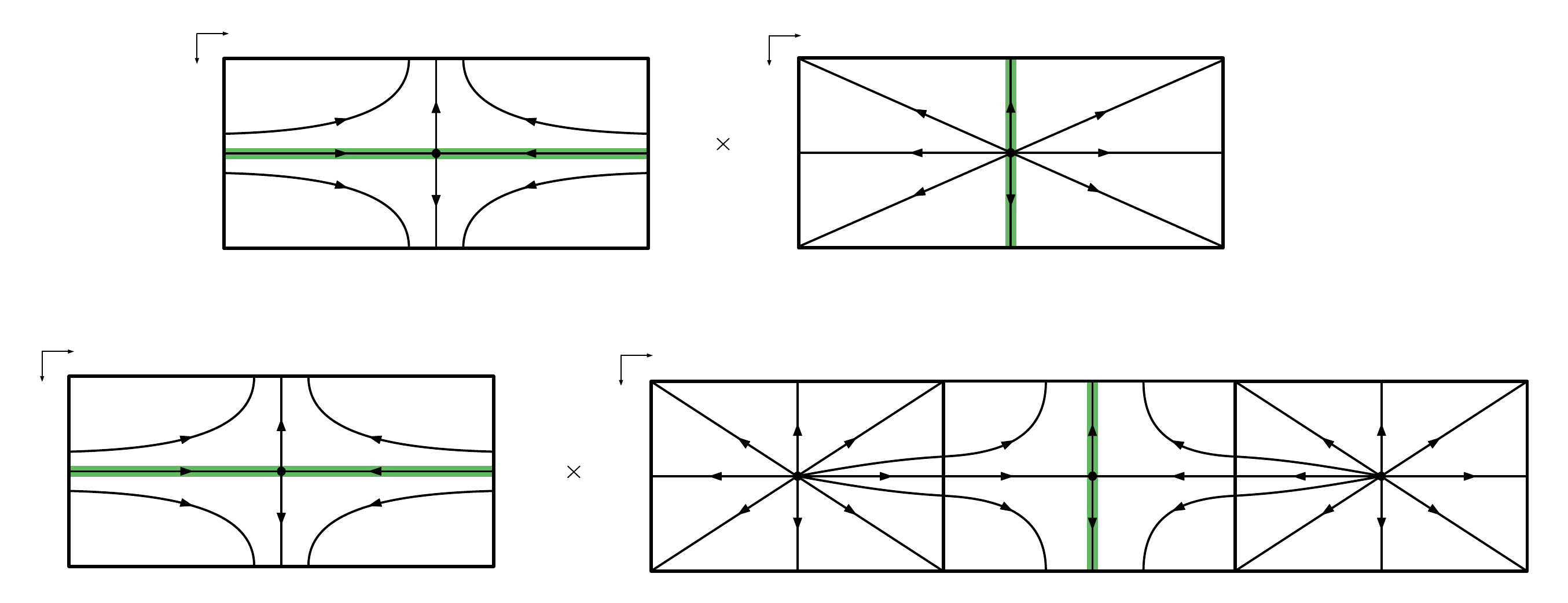}   
        \put(2,12.75){\tiny $x_1$}
        \put(5.5,15.75){\tiny $y_1$}
        \put(39,12.5){\tiny $x_2$}
        \put(42.25,15.5){\tiny $y_2$}

        \put(15,0){\tiny $(h_1^2, \textcolor{darkgreen}{h_1^1})$}
        \put(48,0){\tiny $(h_0^2, \emptyset)$}
        \put(59.5,0){\tiny $\cup$}
        \put(66.5,0){\tiny $(h_1^2, \textcolor{darkgreen}{h_0^1})$}
        \put(78.5,0){\tiny $\cup$}
        \put(86,0){\tiny $(h_0^2, \emptyset)$}

        \put(11.75,32.75){\tiny $x_1$}
        \put(15,36){\tiny $y_1$}
        \put(48.5,32.75){\tiny $x_2$}
        \put(51.75,36){\tiny $y_2$}

        \put(25,20.25){\tiny $(h_1^2, \textcolor{darkgreen}{h_1^1})$}
        \put(61,20.25){\tiny $(h_0^2, \textcolor{darkgreen}{h_0^1})$}
        
	\end{overpic}
	\caption{The first row depicts a $4$-dimensional coupled Weinstein $1$-handle, and the second row depicts, in the language of \cref{def:coupled_lefschetz_handle}, a $4$-dimensional coupled Lefschetz $1$-handle with the decomposition described by the proof of \cref{lemma:coupled_weinstein_lefschetz_handle_htpy}.}
	\label{fig:lefschetz_handle}
\end{figure}
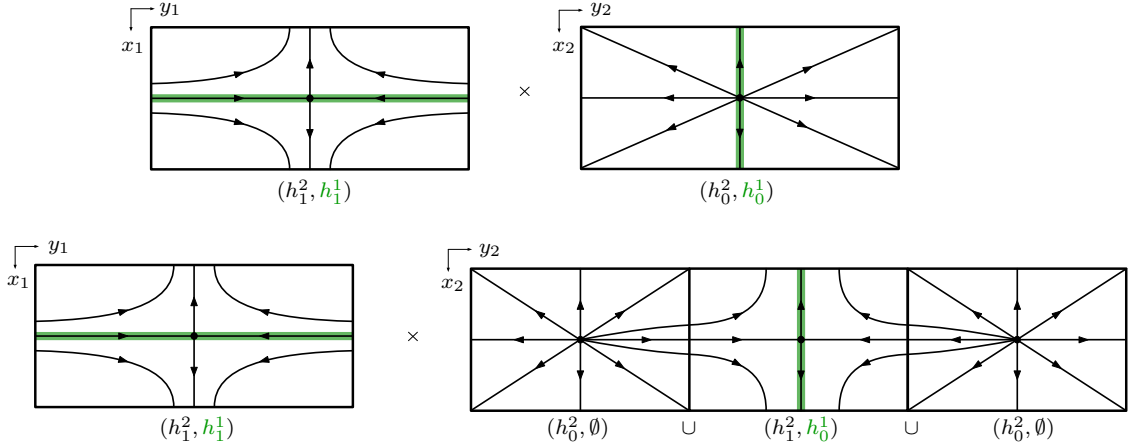

By the proof of \cite[Proposition 2.3]{EliashbergGanatraLazarev2018Regular}, we may perform a $1$-Weinstein homotopy on the pair $(h_0^{2(n-\ell)}, h_0^{n-\ell})$ that births a canceling pair critical points of index $n-\ell-1$ and $n-\ell$ so that the Lagrangian disk $h_0^{n-\ell}$ becomes the unstable manifold of the point of index $n-\ell$. Furthermore, this homotopy can be taken to be everywhere tangent to $h_0^{n-\ell}$ and compactly supported in the interior of $h_0^{2(n-\ell)}$. This yields a coupled Weinstein handle decomposition 
\begin{equation}\label{eq:decomp2}
\left(h_0^{2(n-\ell)}, h_0^{n-\ell}\right) = \left(h_0^{2(n-\ell)}, \emptyset\right) \cup \left(h_{n-\ell-1}^{2(n-\ell)}, \emptyset\right) \cup \left(h_{n-\ell}^{2(n-\ell)}, h_0^{n-\ell}\right).
\end{equation}
For example, see the second term in the product in the second row of \cref{fig:lefschetz_handle}. Combining \eqref{eq:decomp1} and \eqref{eq:decomp2}, we obtain a decomposition (up to $1$-Weinstein homotopy) of the original coupled Weinstein $\ell$-handle as 
\begin{align*}
    \left(h_\ell^{2n}, h_\ell^n\right) &= \left(h_\ell^{2\ell},h_\ell^\ell\right) \times \left(h_0^{2(n-\ell)}, h_0^{n-\ell}\right) \\
    &= \left(h_\ell^{2\ell} \times h_0^{2(n-\ell)},\,\emptyset\right) \cup \left(h_\ell^{2\ell}\times h_{n-\ell-1}^{2(n-\ell)},\, \emptyset\right) \cup \left(h_\ell^{2\ell} \times h_{n-\ell}^{2(n-\ell)}, \, h_\ell^\ell\times h_0^{n-\ell}\right) \\
    &= \left(h_\ell^{2n}, \emptyset\right) \cup \left(h_{n-1}^{2n},\emptyset\right) \cup \left(h_n^{2n}, h_\ell^n\right).
\end{align*}
\end{proof}

\cref{lemma:coupled_weinstein_lefschetz_handle_htpy} justifies the following definition. 

\begin{definition}\label{def:coupled_lefschetz_handle}
Fix $0\leq \ell \leq k \leq n$. 
\begin{enumerate}
    \item When $0\leq \ell \leq n-1$, a $2n$-dimensional \emph{coupled Lefschetz $\ell$-handle} is a sequence of three coupled Weinstein handle attachments 
\[
(h_\ell^{2n}, \emptyset) \cup (h_{n-1}^{2n}, \emptyset) \cup (h_{n}^{2n}, h_\ell^n)
\]
which is $1$-Weinstein homotopic (in particular, diffeomorphic) to a coupled Weinstein $\ell$-handle $(h_\ell^{2n}, h_\ell^n)$ through a homotopy everywhere tangent to the Lagrangian handle $h_\ell^n$ and compactly supported in the interior of $h_\ell^{2n}$. 

\item When $\ell=n$, a $2n$-dimensional \emph{coupled Lefschetz $n$-handle} is simply a coupled Weinstein $n$-handle $(h_{n}^{2n}, h_n^n)$.

\item When $0\leq k \leq n$, a \emph{trivial coupled Lefschetz $k$-handle} is the trivial coupled Weinstein $k$-handle $(h_k^{2n}, \emptyset)$, i.e. just an ambient Weinstein handle.
\end{enumerate}
For $\ell \leq n-1$ we will use the notation $(\mathfrak{h}_\ell^{2n}, h_{\ell}^n)$ to refer to the full data of a coupled Lefschetz $\ell$-handle. In particular, $\mathfrak{h}_\ell^{2n} = h_\ell^{2n} \cup h_{n-1}^{2n} \cup h_n^{2n}$ when $0\leq \ell\leq n-1$. When convenient for notational consistency, we will also let $\mathfrak{h}_n^{2n} = h_n^{2n}$ in the critical case.
\end{definition}

\begin{definition}
A \emph{strict Lefschetz coupled handlebody} is a pair $(W^{2n},L^n)$ where $L\subseteq (W, \lambda, \phi)$ is a Lagrangian in a Weinstein domain built out of a single non-trivial coupled Lefschetz $0$-handle followed by a sequence of non-trivial coupled Lefschetz handle attachments. In a \emph{(non-strict) Lefschetz coupled handlebody}, we further allow trivial coupled Lefschetz handle attachments, i.e. additional ambient Weinstein handles. Finally, we say that a given pair $(W,L)$ admits a \emph{(strict) coupled Lefschetz handle decomposition} if there is a $1$-Weinstein homotopy of $W$ along which $L$ remains Lagrangian such that $(W,L)$ is a (strict) coupled Lefschetz handlebody. 
\end{definition}

\begin{remark}
A strict Lefschetz coupled handlebody is $1$-Weinstein homotopic to $(\D^*L, L)$, where $L$ is the manifold built from the underlying Lagrangian handles in the decomposition.     
\end{remark}

The following proposition characterizes regularity of Lagrangians in terms of the language introduced in this section. 

\begin{proposition}\label{prop:regular_coupled_lefschetz_handlebody}
Let $L^n\subseteq (W^{2n}, \lambda, \phi)$ be a properly embedded regular Lagrangian submanifold of a Weinstein domain. Then for any choice of handle decomposition on $L$, the pair $(W,L)$ admits a coupled Lefschetz handle decomposition extending the prescribed decomposition of $L$. 
\end{proposition}

\begin{proof}
By \cref{prop:regularity_characterization}, we may assume after a $1$-Weinstein homotopy of $(\lambda, \phi)$ that $W = \D^*L \cup W_L$ where $\D^*L$ is the standard Weinstein structure on the cotangent bundle, $W_L$ is a complementary Weinstein cobordism attached to $\partial \D^*L$, and $L\subseteq W$ is the $0$-section of the cotangent bundle. 

Choose a Morse function on $L$ along with a gradient-like vector field, which specifies a smooth handle decomposition of $L$. By \cref{lemma:cot_bund_str}, we may take the Weinstein structure on $\D^*L$ to be one which restricts to this Morse structure on $L$ as the $0$-section. In particular, each smooth $\ell$-handle $h_{\ell}^n$ in the decomposition of $L$ then has a neighborhood given by $\D^*h_{\ell}^n$ with the Weinstein structure of an index-$\ell$ Weinstein handle $h^{2n}_{\ell}$. That is, each $\ell$-handle $h_{\ell}^n$ in $L$ has a neighborhood in $\D^*L$ given by a coupled Weinstein $\ell$-handle $(h_\ell^{2n}, h_\ell^n)$. By \cref{lemma:coupled_weinstein_lefschetz_handle_htpy}, we may perform a further $1$-Weinstein homotopy on each coupled Weinstein handle to witness it as a coupled Lefschetz handle. This presents $(\D^*L, L)$ as a strict coupled Lefschetz handlebody.

The complementary cobordism $W_L$ is attached via a sequence of Weinstein handles, which are all trivial coupled Lefschetz handles. Thus, attaching these handles presents $(W=\D^*L \cup W_L, L)$ as a coupled Lefschetz handlebody. 
\end{proof}

\section{Proof of main theorem}\label{sec:proof}

\begin{definition}\label{def:arc_adm}
An \emph{arc-admissible triple} is the data $(W, L, f)$ where $(W^{2n}, \lambda, \phi)$ is a Weinstein domain, $L^n\subseteq W$ is a properly embedded Lagrangian submanifold, and $f:L \to [0,1]$ is an efficient Morse function; see \cref{def:efficient}. 
\end{definition}

\begin{definition}
Let $(W,L,f)$ be an arc-admissible triple. We say that a Weinstein Lefschetz fibration $p:(W,\lambda, \phi) \to \D^2$ is \emph{admissible for $(W,L,f)$}, or simply \emph{$p:(W,L,f)\to \D^2$ is admissible}, if the following properties hold: 
\begin{enumerate}
    \item The image $p(L)\subseteq \D^2$ is a smoothly embedded arc admitting a parametrization $\gamma:[0,1] \to p(L)$ such that $\gamma^{-1}\circ p:L \to [0,1]$ is a Morse function $0$-homotopic to $f:L \to [0,1]$. 
    \item For any regular value $z\in p(L) \subseteq \D^2$ of $p$, the intersection $L_0^{n-1}:= L\cap p^{-1}(z)$ is a regular Lagrangian in the corresponding regular fiber $p^{-1}(z) \cong W_0^{2n-2}$ of the fibration.
\end{enumerate}

We will also say that $L$ is an \emph{arc-admissible} Lagrangian if such a $p$ exists.
\end{definition}

With the language established above, we may rephrase \cref{thm:main} as follows. 

\begin{theorem}\label{thm:main_alt}
Let $(W,L,f)$ be an arc-admissible triple with $L$ a regular Lagrangian. Then $(W,L,f)$ admits an admissible Weinstein Lefschetz fibration.     
\end{theorem}

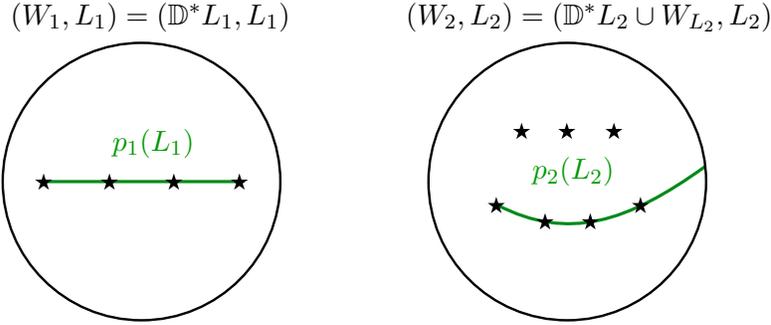
\begin{figure}[ht]
	\begin{overpic}[scale=.3]{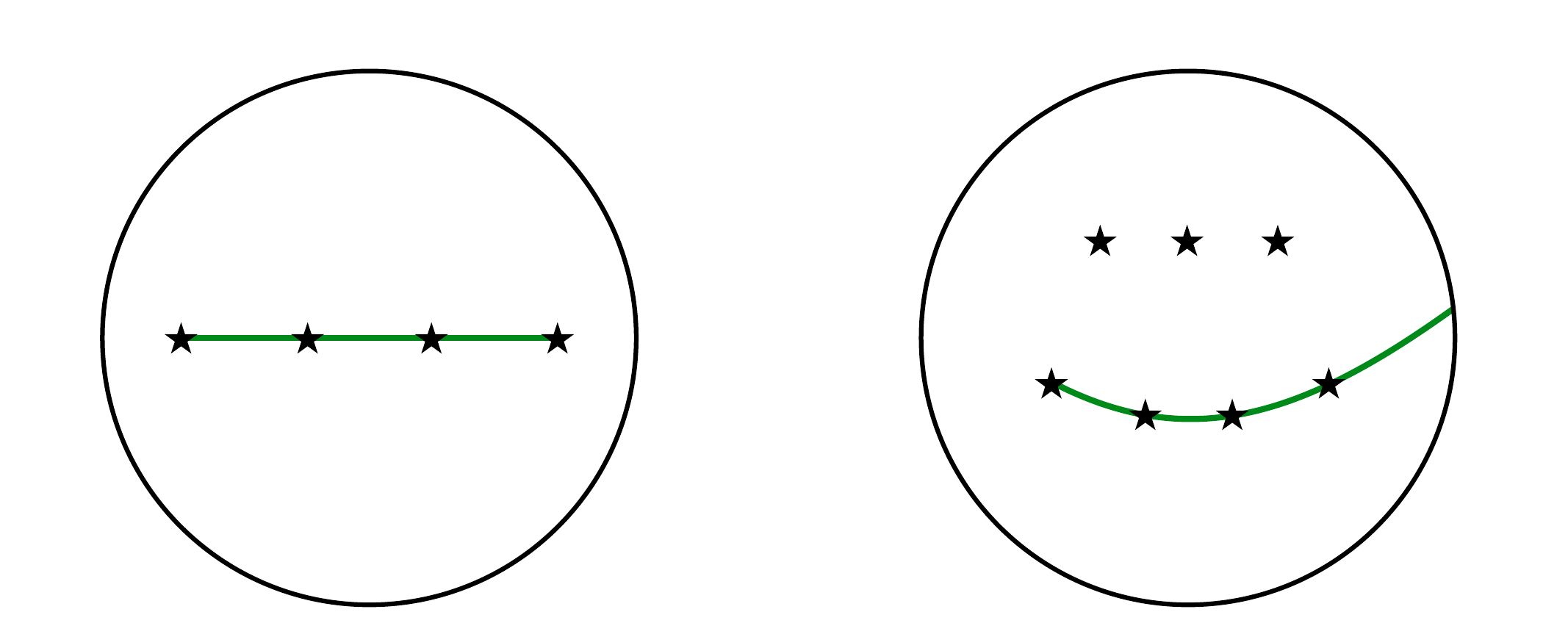}   
        \put(7.5,37.5){\small $(W_1,L_1) = (\D^*L_1, L_1)$}
        \put(20,22){\small \textcolor{darkgreen}{$p_1(L_1)$}}

        \put(56,37.5){\small $(W_2,L_2) = (\D^*L_2 \cup W_{L_2}, L_2)$}
        \put(71.5,18.5){\small \textcolor{darkgreen}{$p_2(L_2)$}}
	\end{overpic}
	\caption{The statement of \cref{thm:main_alt} for two different coupled Lefschetz handlebodies. In both figures, the disk is the base of the Weinstein Lefschetz fibration $p_i:W_i \to \D^2$ and stars are critical values. On the left, $L_1$ is a closed manifold (for instance, diffeomorphic to $T^2$) and $(W_1,L_1)$ is a strict coupled Lefschetz handlebody. On the right, $L_2$ is a manifold with boundary and $(W_2,L_2)$ is a (non-strict) coupled Lefschetz handlebody.}
	\label{fig:coupled_fibration}
\end{figure}

To prove \cref{thm:main_alt}, we will inductively show that every coupled Lefschetz handlebody $(W,L)$ admits an admissible Weinstein Lefschetz fibration. Admissibility refers only to a prescribed Morse function on $L$, and not a handle decomposition (which requires the extra data of a gradient-like vector field), but \cref{prop:regular_coupled_lefschetz_handlebody} --- which relies on the flexibility of \cref{lemma:cot_bund_str} --- allows us to prescribe any handle decomposition compatible with the Morse function on $L$ and extend this to a coupled Lefschetz handle decomposition on $(W,L)$.

The base case is that of a coupled Lefschetz $0$-handle.

\begin{lemma}[Coupled $0$-handles]\label{lemma:0handles}
The triple $(\mathfrak{h}_0^{2n}, h_0^n, f)$ admits an admissible Weinstein Lefschetz fibration, where $(\mathfrak{h}_0^{2n}, h_0^n)$ is a coupled Lefschetz $0$-handle and $f:h_0^n \to [0,1]$ is a Morse function with one critical point of index $0$.      
\end{lemma}

\begin{proof}
This is a consequence of \cite[Proposition 2.3]{EliashbergGanatraLazarev2018Regular}; see also (2) of \cref{example:coupled_handles}. For the sake of completeness we describe the fibration. A coupled Lefschetz $0$-handle is a sequence of three coupled Weinstein handles:
\[
(\mathfrak{h}_0^{2n}, h_0^n) = (h_0^{2n}, \emptyset) \cup (h_{n-1}^{2n}, \emptyset) \cup (h_n^{2n}, h_0^{n})
\]
where the Lagrangian $0$-handle $h_0^{n}$ is the co-core of the critical Weinstein handle $h_n^{2n}$. The subcritical Weinstein handles $h_0^{2n} \cup h_{n-1}^{2n}$ admit the trivial Weinstein Lefschetz fibration with no critical points and regular fiber
\[
W_0^{2n-2}:= h_0^{2n-2} \cup h_{n-1}^{2n-2} = \D^*S^{n-1}.
\]
The attaching sphere of the critical Weinstein handle $h_n^{2n}$ is a Legendrian lift of the Lagrangian $0$-section of a regular fiber. Thus, we attach $h_n^{2n}$ and extend the Weinstein Lefschetz fibration to $p:\mathfrak{h}_0^{2n} \to \D^2$ by attaching a Lefschetz thimble along the $0$-section, making the latter a vanishing cycle. The Lagrangian co-core $h_0^{n}$ of $h_n^{2n}$ is then the "dual" Lefschetz thimble; see the left side of \cref{fig:handle_fibration}. This gives $p(h_0^n)$ as a smoothly embedded arc in the base and moreover the fibration induces the Morse function $f:h_0^n\to [0,1]$, up to $0$-homotopy. 

Finally, observe that if $z\in p(h_0^{n})$ is regular then $h_0^n \cap p^{-1}(z)$ is the $0$-section of $W_0 \cong \D^*S^{n-1}$ and hence is regular. In particular, the open book decomposition of $\partial \mathfrak{h}_0^{2n} = (S^{2n-1}, \xi_{\mathrm{st}})$ induced by the fibration is $(\D^*S^{n-1}, \tau)$, where $\tau$ is the positive Dehn twist around the $0$-section, and the boundary of the Lagrangian submanifold $\partial h_0^n$ is the $0$-section of one of the pages.    
\end{proof}

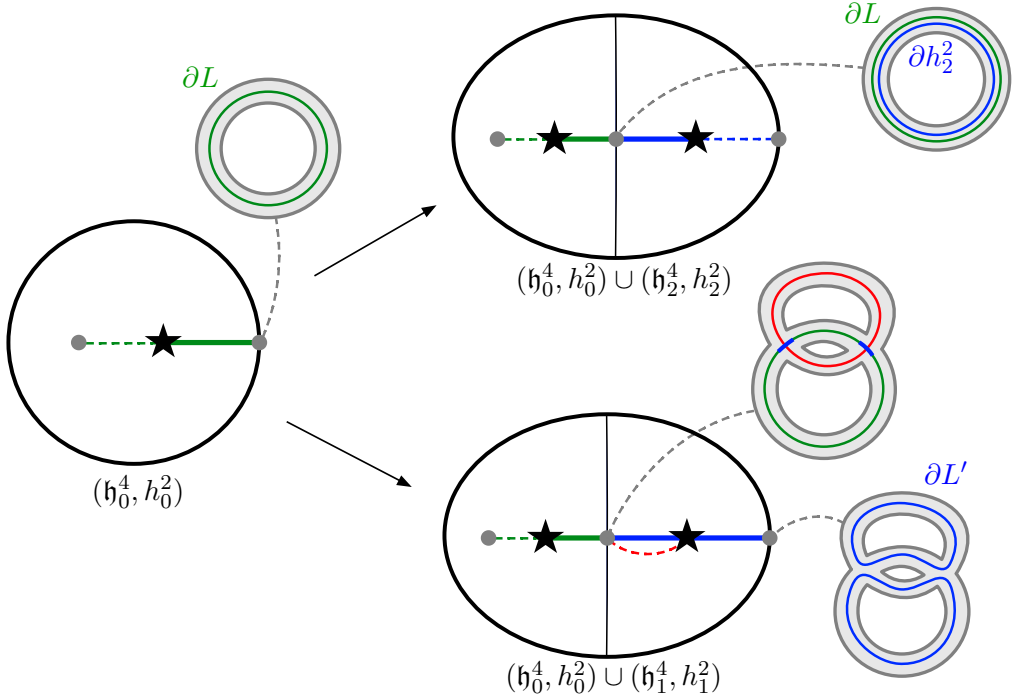
\begin{figure}[ht]
	\begin{overpic}[scale=.63]{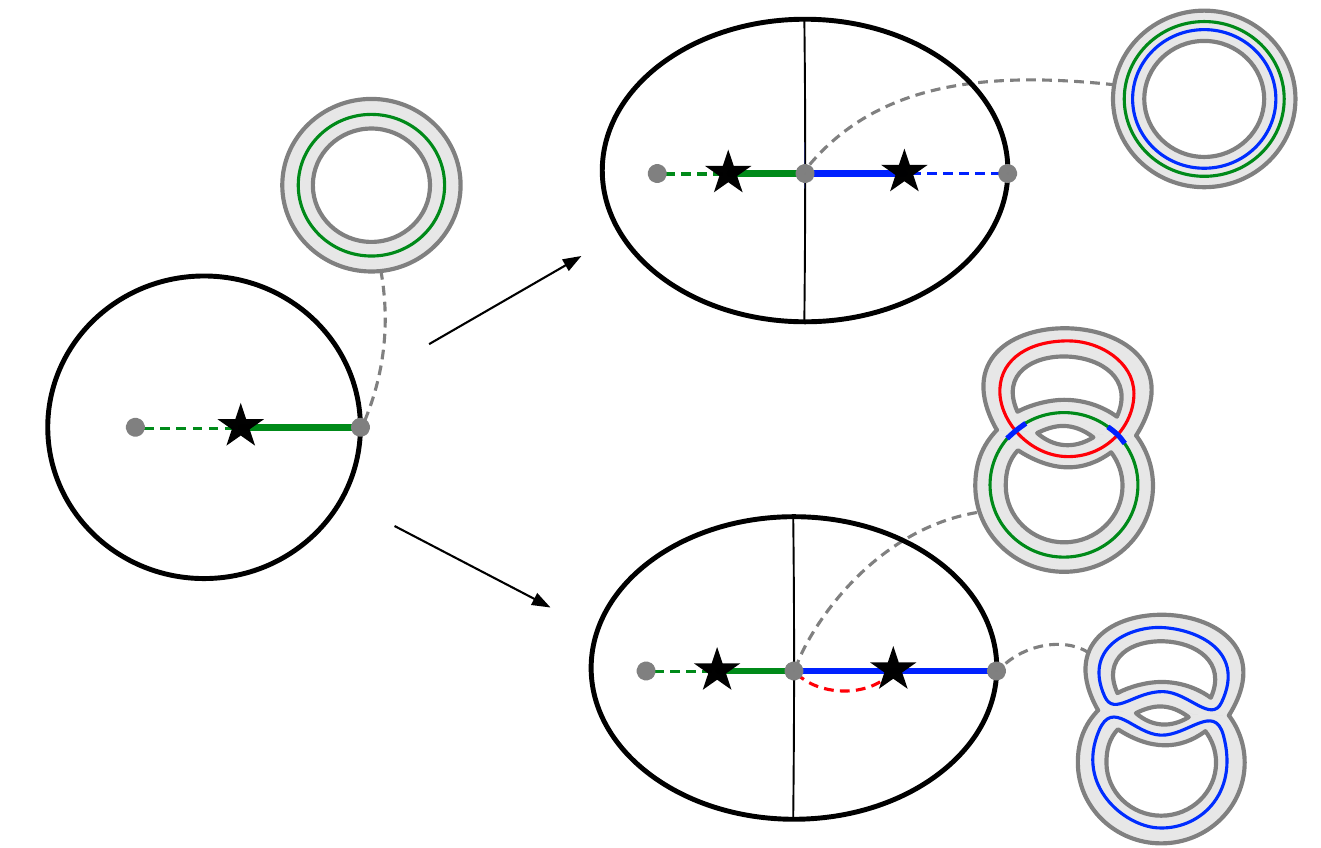}   
       \put(11.5,17){\small $(\mathfrak{h}_0^4, h_0^2)$}
       \put(51.5,37){\small $(\mathfrak{h}_0^4, h_0^2)\cup (\mathfrak{h}_2^4, h_2^2)$}
       \put(50.5,-0.5){\small $(\mathfrak{h}_0^4, h_0^2)\cup (\mathfrak{h}_1^4, h_1^2)$}

       \put(20,56){\small \textcolor{darkgreen}{$\partial L$}}
       \put(82.5,62){\small \textcolor{darkgreen}{$\partial L$}}
       \put(88.5,58){\small \textcolor{blue}{$\partial h_2^2$}}
       \put(90,18.5){\small \textcolor{blue}{$\partial L'$}}
	\end{overpic}
	\caption{Coupled Lefschetz handle attachment in dimension $4$ as described in \cref{lemma:k_handle}. Gray dots are regular values and stars are critical values. The solid colored lines indicate the projection of the Lagrangian submanifold while dashed lines are projections of other Lefschetz thimbles for reference. On the left is a coupled Lefschetz $0$-handle. On the top right we attach a (critical) coupled Lefschetz $2$-handle so that $(W',L')=(\D^*S^2, S^2)$. On the bottom right, we instead attach a (subcritical) coupled Lefschetz $1$-handle. The attaching region of $h_1^2$ is drawn in blue on the central regular fiber, along with the attaching sphere of the critical handle $h_2^4$ which is part of $\mathfrak{h}_1^4 = h_1^4 \cup h_1^4 \cup h_2^4$.}
	\label{fig:handle_fibration}
\end{figure}

We perform the inductive step in two parts, corresponding to attaching either a trivial or non-trivial coupled Lefschetz handle. The following lemma treats the former case.

\begin{lemma}[Trivial coupled Lefschetz handles]\label{lemma:trivial_handle}
Let $p:(W^{2n}, L, f)\to \D^2$ be admissible. Fix $0\leq k \leq n$. Let $(W',L') = (W,L) \cup (h_k^{2n}, \emptyset)$ be obtained from $(W,L)$ by a trivial coupled Lefschetz $k$-handle attachment. Then there is an admissible Weinstein Lefschetz fibration $p':(W',L',f')\to \D^2$. 
\end{lemma}

\begin{proof}
Attaching a trivial coupled Lefschetz handle amounts to attaching an ambient Weinstein handle. To prove the lemma, we must do so in a Lefschetz compatible way so as not to disrupt admissibility. We consider the subcritical case first, then the critical case. 

\vspace{2mm}
\textsc{Case 1:} {\em $0\leq k\leq n-1$.}
\vspace{2mm}

Attaching subcritical handles in a Lefschetz-compatible manner is standard. Namely, if $(h_{k}^{2n}, \lambda_{k}, \phi_{k})$ denotes a $2n$-dimensional Weinstein handle of subcritical index $0\leq k\leq n-1$, the model described in \cref{subsec:coupled} decomposes as $h_{k}^{2n} \cong h_{k}^{2n-2}\times h_0^2$; see also Cieliebak's more general fact that subcritical Weinstein manifolds are split:  \cite{Cieliebak2002Split} or \cite[\S 14.4]{CieliebakEliashberg2012FromStein}. 
The horizontal boundary of $W$ is $\partial W_0 \times \D^2$, where $W_0$ denotes the regular fiber of the fibration. After an additional Weinstein homotopy, we may attach the subcritical handle $h_{k}^{2n} \cong h_{k}^{2n-2}\times \D^2$ along the horizontal boundary by identifying the $\D^2$-direction with $h_0^2$. At the fibration level, this only has the effect of attaching the handle $h_{k}^{2n-2}$ to the regular fiber --- modifying the fibration away from $L$ --- and has no impact on admissibility.    

\vspace{2mm}
\textsc{Case 2:} {\em $k=n$.}
\vspace{2mm}

Let $\Lambda\subset \partial W$ be the Legendrian attaching sphere of $h_n^{2n}$. After a further isotopy of the attaching spheres if necessary (e.g. induced by the horizontal Liouville flow) we may assume that $\Lambda \subseteq W_0 \times \partial \D^2$ is embedded in the vertical boundary. 

By \cref{thm:obd_exist_leg_page}, there is a strongly Weinstein open book decomposition $(B_1, \pi_1)$ of $\partial W$ with Weinstein page $W_1$ such that $\Lambda\subseteq W_1$ embeds as a regular Lagrangian on a single page. Let $(B, \pi)$ denote the strongly Weinstein open book decomposition of $\partial W$ induced by the given Lefschetz fibration $p$. 

By the Giroux correspondence (\cref{thm:GC}), $(B_1,\pi_1)$ and $(B,\pi)$ admit a common multistabilization $(B',\pi')$ with Weinstein page $W_0'$. Since $(B, \pi)$ is the open book induced by the original Lefschetz fibration, the sequence of stabilizations bringing $(B, \pi)$ to $(B', \pi')$ is induced by a sequence of Lefschetz fibration stabilizations. We claim that these stabilizations preserve admissibility. At the Weinstein level, a Lefschetz fibration stabilization involves a subcritical handle $\tilde{h}_{n-1}^{2n} = \tilde{h}_{n-1}^{2n-2}\times \D^2$, witnessed as a critical handle attached to the regular fiber, followed by a critical handle $\tilde{h}_{n}^{2n}$ attached along a vanishing cycle. As before, the subcritical handle attachment does not affect admissibility of the resulting fibration. After a generic perturbation if necessary, we may assume the attaching sphere of the critical handle (the new vanishing cycle) avoids $\partial L$ if the latter is non-empty. Since $\partial L$ lies in a page of the open book, the critical handle is moreover disjoint from $L$. The extension of the fibration across the critical handle attachment thus avoids $L$ and also preserves admissibility.

Finally, we attach the original critical handle $h_n^{2n}$ to $W$ using the stabilized fibration with regular fiber $W_0'$. Since $\Lambda \subset W_1^{2n-2}$, and the page $W_0'$ was obtained from $W_1$ by attaching $(2n-2)$-dimensional Weinstein handles, we have $\Lambda \subset W_0'$ embedded in the new regular fiber as a regular Lagrangian. Thus, to attach the critical handle $h_n^{2n}$ to $W$ to obtain $W'$ with  a Lefschetz fibration, we simply include $\Lambda$ as an additional vanishing cycle in the stabilized Lefschetz fibration of $W$ with page $W_0'$. As with the critical handles involved in stabilization, we may assume this affects the fibration away from the Lagrangian $L$ and thus preserves admissibility.
\end{proof}

\begin{remark}
\cref{lemma:trivial_handle} is the only place where we appeal to the Giroux correspondence and convex hypersurface theory.    
\end{remark}

Now we consider non-trivial coupled Lefschetz handles. 

\begin{lemma}\label{lemma:arc_image_lemma}
Let $p:W^{2n} \to \D^2$ be a Weinstein Lefschetz fibration. Let $h_n^{2n}$ denote one of the critical handles comprising the Weinstein structure on $W$ whose attaching sphere is one of the vanishing cycles of the fibration. Assume moreover that the vanishing cycle is a regular Lagrangian in the regular fiber. If $(h_n^{2n},h_\ell^n)$ is a coupled Weinstein $(n,\ell)$-handle, then up to Weinstein homotopy, $p(h_\ell^n)\subset \D^2$ is an embedded arc admitting a parametrization $\gamma:[0,1]\to p(h_\ell^n)$ such that $\gamma^{-1}\circ p$ is $0$-homotopic to the index-$\ell$ Morse function on $h_\ell^n$. Moreover, for any regular value $z\in p(h_\ell^n)$, the intersection $L_0 = h_\ell^n \cap p^{-1}(z)$ is a regular Lagrangian in the regular fiber $W_0$.
\end{lemma}

\begin{remark}\label{remark:thimble_arc}
When $\ell=0,n$, \cref{lemma:arc_image_lemma} describes a Lefschetz thimble and its arc projection.    
\end{remark}

\begin{proof}
In a $2n$-dimensional coupled Weinstein $(n,\ell)$-handle, the Lagrangian $\ell$-handle $h_\ell^n$ sits inside the Weinstein $n$-handle as 
\[
h_{\ell}^n = \D^\ell_{(y_1, \dots, y_\ell)} \times \D^{n-\ell}_{(x_{\ell+1}, \dots, x_n)} \subset \D^n_{(y_1, \dots, y_n)} \times \D^n_{(x_1, \dots, x_n)} = h_n^{2n}
\]
given by $\{y_{\ell+1} = \cdots = y_n = x_1 = \cdots = x_\ell = 0\}$. Near a critical point of a Lefschetz fibration, there are local complex coordinates so that the fibration is of the form $p(z_1,...,z_n) = \sum_{j=1}^{n} z_j^2$. We may take $z_j = x_j + i\, y_j$ relative to the Weinstein $n$-handle. Then
\begin{align*}
   p\mid_{h_\ell^n}(z_1,\dots,z_n) &= \sum_{j=1}^\ell (iy_j)^2 + \sum_{j=\ell+1}^n x_j^2 = -\sum_{j=1}^\ell y_j^2 + \sum_{j=\ell+1}^n x_j^2. 
\end{align*}
Clearly $p(h_\ell^n)\subset \{\mathrm{Im}(z) = 0\}\subset \C$. Assuming $\ell\neq 0,n$ for simplicity (c.f. \cref{remark:thimble_arc}), the attaching sphere $S^{\ell-1}_{(y_1,\dots,y_\ell)} \times \{x_{\ell+1} = \cdots = x_n =0\}$ and the belt sphere $\{y_{1} = \cdots = y_\ell =0\} \times S^{n-\ell-1}_{(x_{\ell+1},\dots,x_n)}$ are mapped to $-1$ and $+1$, respectively. Thus, $p(h_\ell^n)$ is an arc, and the fibration $p$ induces the model $\ell$-handle function $\sum_{i=\ell+1}^n x_j^2 - \sum_{j=1}^\ell y_j^2$ when restricted to $h_{\ell}^n$. 

To verify the last claim about regular intersections in the regular fiber, consider the attaching region of the Lagrangian $k$-handle (the co-attaching region following from a similar argument), which sits inside the attaching region of the ambient Weinstein $n$-handle attaching region as
\[
S^{k-1}_{(y_1,\dots,y_k)} \times \D^{n-k}_{(x_{k+1}, \dots, x_n)} \subset S^{n-1}_{(y_1,\dots,y_n)} \times \D^n_{(x_1,\dots,x_n)}.
\]
Identify via a contactomorphism the latter region with the (disk) jet bundle 
\[
S^{n-1}_{(y_1,\dots,y_n)} \times \D^n_{(x_1,\dots,x_n)} \cong J^1S^{n-1}_{(y_1,\dots,y_n)} \cong [-1,1]\times \D^*S^{n-1}.
\]
Since the vanishing cycle $S^{n-1}_{(y_1,\dots,y_n)}\times \{0\}$ is a regular Lagrangian in the regular fiber by assumption, we may thus locally model the Weinstein structure of the regular fiber $W_0$ as the cotangent bundle of $S^{n-1}_{(y_1,\dots,y_n)}$. Here, the Lagrangian handle restricts to the conormal bundle of a smoothly embedded $S^{k-1}\subset S^{n-1}$. Conormal bundles of smooth submanifolds of the $0$-section of a cotangent bundle are regular. 
\end{proof}

\begin{lemma}[Non-trivial coupled Lefschetz handles]\label{lemma:k_handle}
Let $p:(W^{2n}, L, f)\to \D^2$ be admissible and assume that $\partial L \neq 0$. Fix $1\leq k \leq n$. Let $(W',L') = (W,L) \cup (\mathfrak{h}_\ell^{2n}, h_\ell^n)$ be obtained from $(W,L)$ by a coupled Lefschetz $\ell$-handle attachment and extend $f:L \to \R$ to $f':L' = L\cup h_\ell^n \to \R$ by the model Morse function on $h_\ell^n$ with a single critical point of index $\ell$. Then there is an admissible Weinstein Lefschetz fibration $p':(W',L',f')\to \D^2$. 
\end{lemma}

\begin{proof}
Let $W_0^{2n-2}$ be the regular fiber of the fibration $p:(W,L,f)\to \D^2$. By admissibility, $\partial L\subseteq W_0$ is a regular (closed) Lagrangian, so up to Weinstein homotopy of the regular fiber we may write $W_0 = \D^*(\partial L) \cup W_{\partial L}$ where $W_{\partial L}$ is a complementary Weinstein cobordism. We consider the subcritical and critical cases separately.     

\vspace{2mm}
\textsc{Case 1:} {\em $1\leq \ell\leq n-1$.}
\vspace{2mm}

Recall that a coupled Lefschetz $\ell$-handle is a sequence of three handle attachments 
\[
(\mathfrak{h}_\ell^{2n}, h_\ell^n) = (h_\ell^{2n}, \emptyset) \cup (h_{n-1}^{2n}, \emptyset) \cup (h_n^{2n}, h_\ell^n).
\]
We attach the first two handles to $W$ and extend the fibration in a non-singular way so that the regular fiber is $\tilde{W}_0 := \D^*(\partial L) \cup W_{\partial L} \cup h_\ell^{2n-2} \cup h_{n-1}^{2n-2}$. It remains to attach the coupled Weinstein $(n,\ell)$-handle $(h_n^{2n}, h_\ell^n)$, verify that the vanishing cycle associated to the attaching sphere of $h_n^{2n}$ is a regular Lagrangian in the regular fiber of the fibration, and then apply \cref{lemma:arc_image_lemma} to give admissibility of the resulting Lefschetz fibration. 

We first verify that the attaching sphere of $h_n^{2n}$ is (a Legendrian lift of) an embedded regular Lagrangian sphere in the new regular fiber $\tilde{W}_0$, along which we attach the corresponding Lefschetz thimble. By definition of a coupled Lefschetz $\ell$-handle, $h_n^{2n}$ cancels $h_{n-1}^{2n-2}$, so we may assume that the attaching sphere of $h_n^{2n}$ is $D\cup C$, where $C$ is the core of $h_{n-1}^{2n-2}$ in the regular fiber and $D$ is some disk. Thus, it suffices to argue that $D$ is a regular Lagrangian disk in $W_0 \cup h_\ell^{2n-2}$.

Let $\sigma \cong S^{\ell-1} \subseteq \partial L$ be the attaching sphere of the Lagrangian $k$-handle $h_\ell^n$, and let $\N^*\sigma \subseteq \D^*(\partial L)$ denote the (disk) conormal bundle. Note that $\N^*\sigma\cong S^{\ell-1}\times \D^{n-\ell}$ is a Lagrangian submanifold of $\D^*(\partial L)$. Inside the Weinstein handle $h_\ell^{2n-2}$ that was attached to the original regular fiber $W_0$ there is a model Lagrangian $\ell$-handle $h_\ell^{n-1}\cong \D^\ell \times \D^{n-\ell-1}\subseteq h_\ell^{2n-2}$ tangent to the Weinstein structure. Up to Weinstein homotopy, the Weinstein handle $h_\ell^{2n-2}$ was attached to $W_0 = \D^*(\partial L) \cup W_{\partial L}$ along the boundary of the conormal bundle $\partial \N^*\sigma \subseteq \partial \D^*(\partial L)$, so that inside $\D^*(\partial L) \cup h_\ell^{2n-2}$ there is a properly embedded exact Lagrangian submanifold 
\begin{equation}\label{eq:subcritical_proof}
\N^*\sigma \cup h_{\ell}^{n-1} \cong (S^{\ell-1}\times \D^{n-\ell}) \cup (\D^\ell \times \D^{n-\ell-1})    
\end{equation}
where $h_\ell^{n-1} \cong \D^\ell \times \D^{n-\ell-1}$ is attached along $S^{\ell-1}\times \{q\}$ for some $q\in \partial \D^{n-\ell}$. As conormal bundles are regular, it follows that $\N^*\sigma \cup h_{\ell}^{n-1}\subset \D^*(\partial L) \cup h_\ell^{2n-2}$ is regular. Note that we may view the conormal bundle as 
\begin{equation}\label{eq:subcritical_proof2}
\N^*\sigma \cong S^{\ell-1}\times \D^{n-\ell} \cong \left(\D^{\ell-1}\times \D^{n-\ell}\right) \cup h_{\ell-1}^{n-1}
\end{equation}
where we think of $\D^{\ell-1}\times \D^{n-\ell}$ as a $0$-handle. Therefore, the $\ell$-handle $h_\ell^{n-1}\cong \D^\ell \cup \D^{n-\ell-1}$ in \eqref{eq:subcritical_proof} cancels the ($\ell-1$)-handle in \eqref{eq:subcritical_proof2} and thus $\N^*\sigma \cup h_{\ell}^{n-1} \cong D$ is a properly embedded regular Lagrangian disk in $W_0 \cup h_\ell^{2n-2}$, as desired. 

\vspace{2mm}
\textsc{Case 2:} {\em $\ell=n$.}
\vspace{2mm}

In this case we are simply attaching a coupled Weinstein $n$-handle $(h_n^{2n}, h_n^n)$. By assumption, the attaching spheres of the ambient Weinstein handle and the Lagrangian handle coincide. Also by assumption $\partial L\subset W_0$ is a regular Lagrangian in the regular fiber, so it must be the case that (one component of) $\partial L$ is a smoothly embedded $n$-sphere. Attaching a Lefschetz thimble along this sphere and applying \cref{lemma:arc_image_lemma} gives the desired admissible fibration. 
\end{proof}

\begin{proof}[Proof of \cref{thm:main_alt}.]
Given $(W,L,f)$, \cref{prop:regular_coupled_lefschetz_handlebody} implies that $(W,L)$ admits a coupled Lefschetz handle decomposition with a single coupled Lefschetz $0$-handle. By \cref{lemma:0handles}, \cref{lemma:trivial_handle}, and \cref{lemma:k_handle}, an admissible Lefschetz fibration may be built handle by handle. 
\end{proof}

\section{Mutations via arc projections}\label{sec:mutation}

In this section we prove \cref{thm: mutation}, and describe how to visualize embedded regular fillings of Legendrian links in Lefschetz fibrations of the standard Weinstein $4$-ball.

We recall the notions of Weinstein pairs as defined in \cite{Eliashberg2018Weinstein} and of Lagrangian mutation as in \cite{pascaleff-tonkonog, yau_surgery}.

\begin{definition}
    A \emph{Weinstein pair} $(W, \Sigma)$ consists of a Weinstein domain $(W^{2n},\lambda,\phi)$ together with a Weinstein hypersurface $(\Sigma^{2n-2}, \lambda|_{\Sigma}) \subset \partial W$.
\end{definition}

For our purpose, $W$ will be the standard $4$-ball $(B^4, \lambda_{\mathrm{st}}, \phi_{\mathrm{st}})$, $\Sigma$ will be the Weinstein thickening of a Legendrian link $\Lambda \subset (S^3, \xi_{\mathrm{st}})$, and we use the abbreviated notation $(B^4, \Lambda)$ for the Weinstein pair.

\begin{definition}
    Let $L$ be a properly embedded exact Lagrangian in a symplectic manifold $M$, and $D$ be an embedded Lagrangian disk in the interior of $M$ such that $L \cap D = \partial D$ \emph{cleanly}, i.e. $D$ is transverse to $L$ along $\partial D$ and $\partial D$ is an embedded smooth curve in $L$. Such a disk is called an \emph{$\mathbb{L}$-compressing disk}.\footnote{This definition was coined by Casals and Weng in \cite{casals2024microlocal}.} Then, the \emph{mutated Lagrangian} $\mu_D(L)$ is obtained by Lagrangian disk surgery (see \cite{polt,yau_surgery}) of $L$ along $D$.
\end{definition}

It is shown in \cite{yau_surgery} that if $L$ is exact, so is $\mu_D(L)$, and they are smoothly isotopic. In addition, $\mu_D(L)$ has an $\mathbb{L}$-compressing disk $D'$ such that $\mu_{D'}(\mu_D(L)) = L$. The argument of \cite{casals_gao24} and a number of related papers in the area (e.g. \cite{casals-li22,casals2024microlocal,cggs24}) use the notion of an {\em $\mathbb{L}$-compressing system}, which is a maximal linearly independent (in $H_1(L)$) collection of curves that simultaneously bound $\mathbb{L}$-compressing disks.

\begin{definition}\label{def: cal-skeleton}
    A {\em closed arboreal skeleton} (or \emph{CAL-skeleton}, as in \cite{casals_skeleta}) of the Weinstein pair $(B^4, \Lambda)$ \emph{associated to an exact filling $L$} is the data $(L, \mathcal{D}, \Gamma)$ where
    \begin{itemize}
    \item $\mathcal{D}=\{D_i\}$ is a collection of $\mathbb{L}$-compressing disks with boundaries that form an $\mathbb{L}$-compressing system $\Gamma = \{ \gamma_i\}$ for $L$, and 
    \item $L \cup \mathcal{D}$ is arboreal (see \cite{nadler2017arboreal}) and tangent to the Liouville flow of the Weinstein structure. In particular, $L\cup \mathcal{D}$ is regular.
\end{itemize}
\end{definition}

Recall that arc-admissibility of regular Lagrangians requires the notion of an efficient Morse function on the Lagrangian. We incorporate this with CAL-skeleta.

\begin{definition}\label{def: compatible_CAL}
    Let $f:L \to [0,1]$ be an efficient Morse function and $(L,\mathcal{D},\Gamma)$ a CAL-skeleton for the Weinstein pair $(B^4, \Lambda)$. We say $(L,\mathcal{D},\Gamma)$ is {\em compatible} with $(L,f)$ if for all $\gamma \in \Gamma$, $\gamma$ is embedded in the $1$-skeleton of some handle decomposition associated to $f$.
\end{definition}

Given a CAL-skeleton $(L, \mathcal{D}, \Gamma)$, it is natural to ask if one obtain CAL-skeleta after mutations on $L$ along disks in $\mathcal{D}$. Given $D \in \mathcal{D}$, the mutated Lagrangian is $\mu_D(L)$, and the mutated $\mathbb{L}$-compressing system $\mu_D(\Gamma)$ is the result of a $\gamma$-exchange on $\Gamma$ \cite[Section 2.3]{casals_gao24}, where $\gamma = \partial D$. As discussed in \cite[Example 4.10]{casals_gao24}, it can happen that there exists $D$ such that some curves in the $\mathbb{L}$-compressing system become immersed in $\mu_D(\Gamma)$, thus we do not get a CAL-skeleton after mutation along $D$. Despite this, in this section, we show that the property of being arc-admissible (for the same Lefschetz fibration) is preserved under mutations.

\subsection{Arc-admissibility of mutated Lagrangians}

Let $\Lambda \subset (S^3, \xi_{\mathrm{st}})$ be a Legendrian link, with $L$ an exact filling and $(L, \mathcal{D},\Gamma)$ an associated CAL-skeleton. Following \cite[Lemma 2.2]{EliashbergGanatraLazarev2018Regular} as in the proof of \cref{thm:main}, the standard Weinstein structure on $(B^4, \Lambda)$ is homotopic to the one obtained by starting with $T^*L$ (seen as a neighborhood of $L$ inside $B^4$) and then attaching Weinstein $2$-handles. Let $W_L$ denote the cobordism comprising the collection of the latter $2$-handles, whose cores can be regarded as the disks $D_i \in \mathcal{D}$.

\begin{figure}[ht]
	\begin{overpic}[scale=.6]{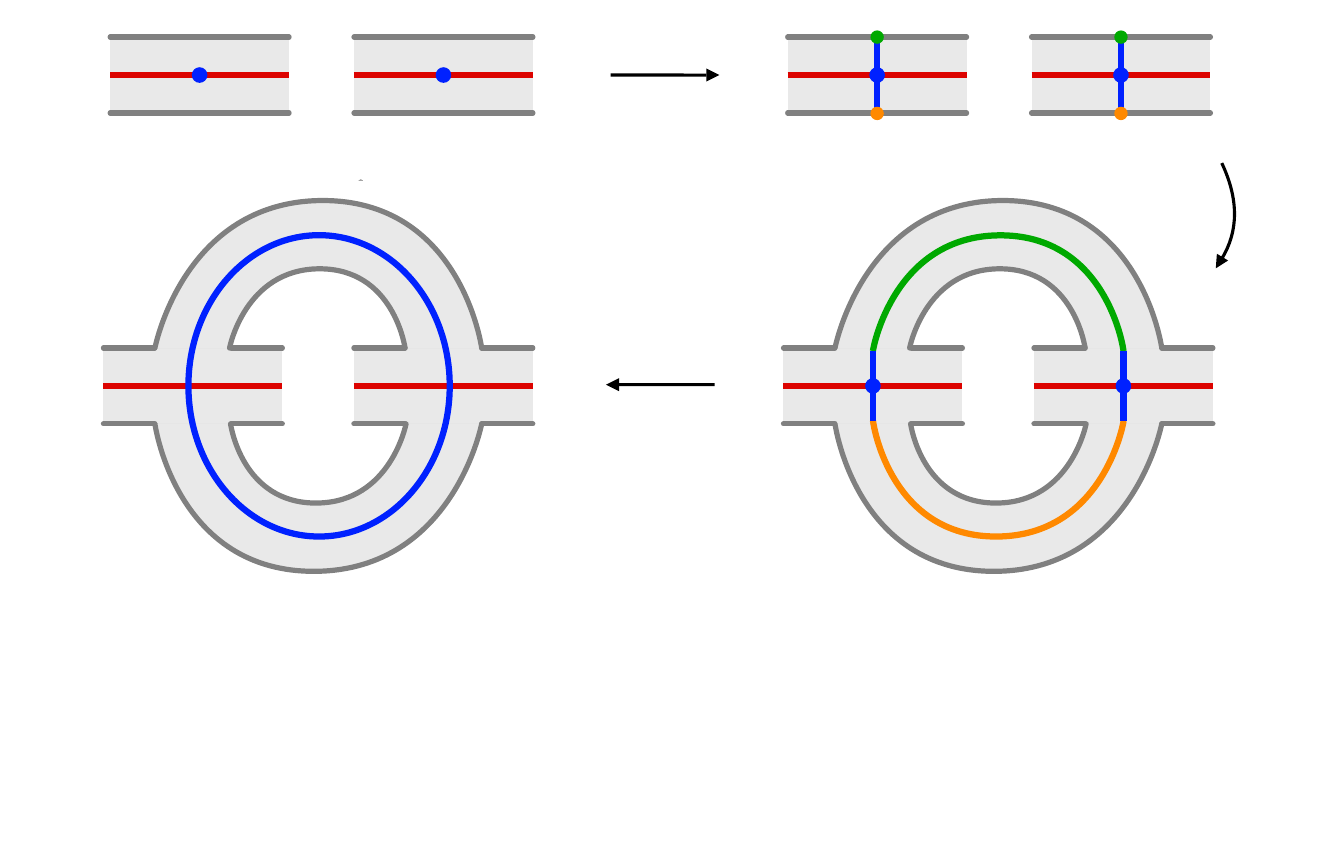}
    \put(3,57.5){\textcolor{darkred}{$V_0$}}
    \put(13,63){\small \textcolor{darkblue}{$S^0$}}
    \put(8,45){\textcolor{darkblue}{$V_{1,i}$}}
    \end{overpic}
    \vskip-2.3cm
    \caption{Attaching a coupled Lefschetz $1$-handle as described in \eqref{part:johns2} of Johns' algorithm. In the top right, the green (resp. orange) $S^0$ is the positive (resp. negative) boundary of the conormal bundle in the regular fiber of the attaching sphere of the Lagrangian $1$-handle.}
    \label{fig:conormal}
    \end{figure}

We first recall how to build an admissible Lefschetz fibration for the $0$-section of $T^*L$, following the argument of \cite{Johns2011Lefschetz} used in \cref{sec:proof}. First, choose a handle decomposition of $L$ with one $0$-handle. Since we assume $\partial L \neq \emptyset$ in this section, we may assume there are no $2$-handles. Beginning with an initial trivial Lefschetz fibration $\D^2 \times \D^2$:
\begin{enumerate}
    \item For the $0$-handle of $L$, add a $1$-handle to the fiber and a vanishing cycle $V_0$ along the core of this $1$-handle. This identifies the $0$-handle of $L$ with the vanishing thimble associated to $V_0$.
   \item For every $1$-handle $h_i$ of $L$, identify the attaching sphere $S^0_i \subset V_0$ and attach a pair of $1$-handles to the regular fiber, one along the positive conormal pushoff of $S^0_i$ and the other along the negative conormal pushoff. Then, include a vanishing cycle $V_{1,i}$ along the union of the cores of the $1$-handles and the conormal bundle of $S^0_i$ in the original regular fiber.\label{part:johns2} 
\end{enumerate}

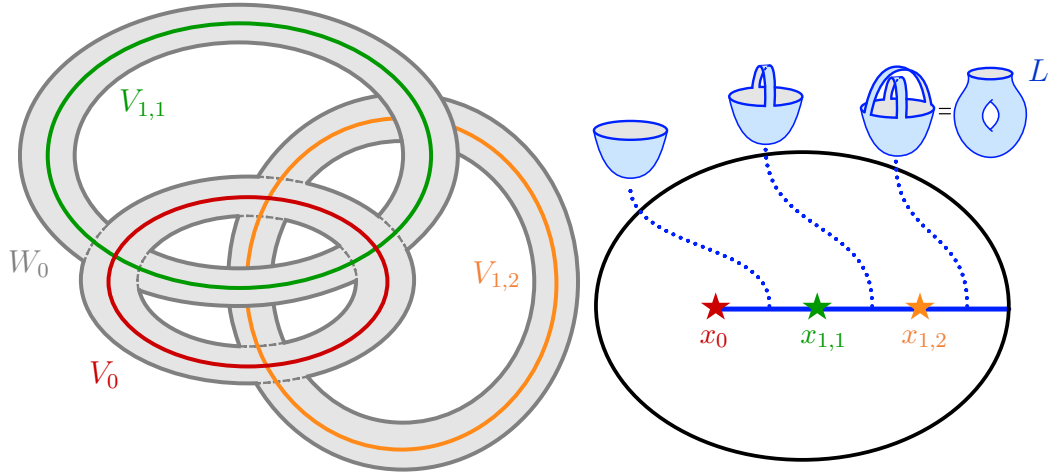
\begin{figure}[ht]
    \vskip-2cm
	\begin{overpic}[scale=.6]{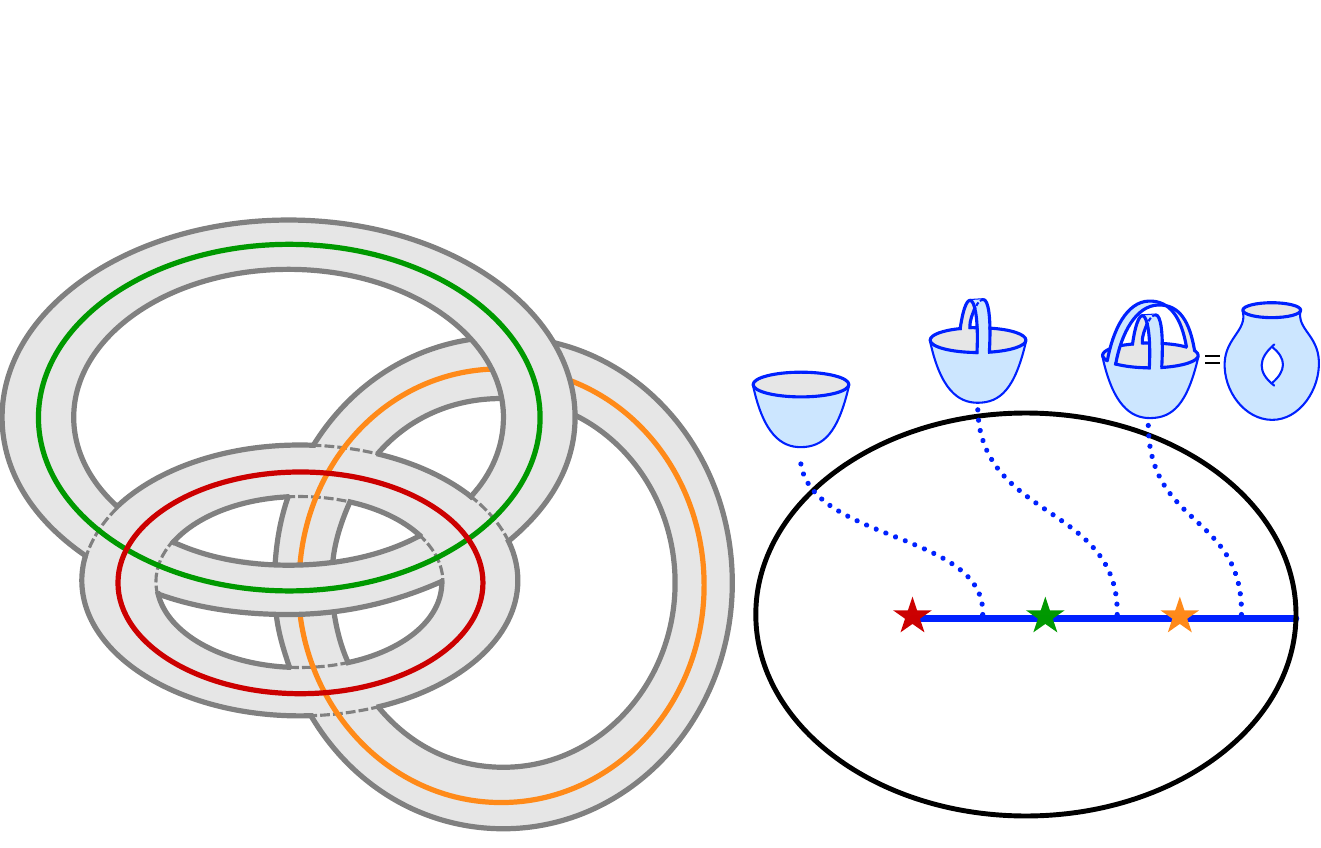}
    \put(-1,21){\textcolor{gray}{$W_0$}}
    \put(7,10){\textcolor{darkred}{$V_0$}}
    \put(10,37){\textcolor{darkgreen}{$V_{1,1}$}}
    \put(45,20){\textcolor{Orange}{$V_{1,2}$}}
    \put(67.5,14){\small \textcolor{darkred}{$x_0$}}
    \put(77.5,14){\small \textcolor{darkgreen}{$x_{1,1}$}}
    \put(87.5,14){\small \textcolor{Orange}{$x_{1,2}$}}
    \put(100,40){\textcolor{darkblue}{$L$}}
    \end{overpic}
    \caption{The admissible Lefschetz fibration for the pair $(T^*L,L)$ when $L$ is a once-punctured torus. The regular fiber $W_0$ of the fibration is in gray.}
    \label{fig: LF_cotangent}
    \end{figure}

 See \cref{fig: LF_cotangent}. In the above Lefschetz fibration, the image of $L$ is a ray starting at the critical point corresponding to $V_0$ and ending at the boundary of the base $\D^2$.
 In the general case, there will be two kinds of critical points for the Lefschetz fibration representing $W$, the first coming from the ones in $T^*L$, and the others coming from the $2$-handles in the cobordism $W_L$. Up to Weinstein homotopies and homotopies through Lefschetz fibrations,  any ray starting at $V_0$ that passes through all the critical points for $T^*L$, and misses all the critical points for $W_L$, can a priori be the image of $L$. We therefore introduce the notion of \emph{standard Weinstein Lefschetz picture} to fix the background Lefschetz data.

 \begin{definition}[Standard Weinstein Lefschetz picture]\label{def: standard_WLp}
     Given a coupled Weinstein structure on $(W,L,f)$ admitting a compatible CAL-skeleton associated to $L$ and an admissible Lefschetz fibration $p$, we can ensure via a Weinstein homotopy that
     \begin{itemize}
         \item the critical values for the vanishing cycles $V_0, V_{1,i}$, denoted as $x_0, x_{1,i}$, in $T^*L$ are arranged along a horizontal line in the base, 
         \item the critical values of the fibration restricted to $W_L$, denoted as $q_1, \dots, q_k$, are arranged along a parallel horizontal line below, and 
         \item the image of $L$ under $p$ is the horizontal line starting at $x_0$ and terminating at the boundary of $\D^2$, passing through all the values $x_{1,i}$.
     \end{itemize}
     We call this the {\em standard Weinstein Lefschetz picture} of the pair $(W,L)$. Furthermore, by a homotopy of the coupled Weinstein Lefschetz handlebody, we may combine all the critical values $\{x_{1,i}\}$ of $p_L$ to a single value $x_1$ whose pre-image contains all index $1$ critical points of $L$. We call this the {\em degenerate standard Weinstein Lefschetz picture} (see \cref{fig: mutation}).
  \end{definition}

  In a standard Weinstein Lefschetz picture, the disks in an $\mathbb{L}$-compressing system for $L$ corresponding to the cores of $W_L$, after a possible Hamiltonian isotopy, project to arcs joining $p(L)$ with the critical values $q_1, \dots, q_k$. Generally, when we fix the $\mathbb{L}$-compressing system, we fix the associated vanishing cycles on the fiber, i.e. we fix the arcs. We now describe the result of mutations on $L$, and prove \cref{thm: mutation} by proving the following restatement.

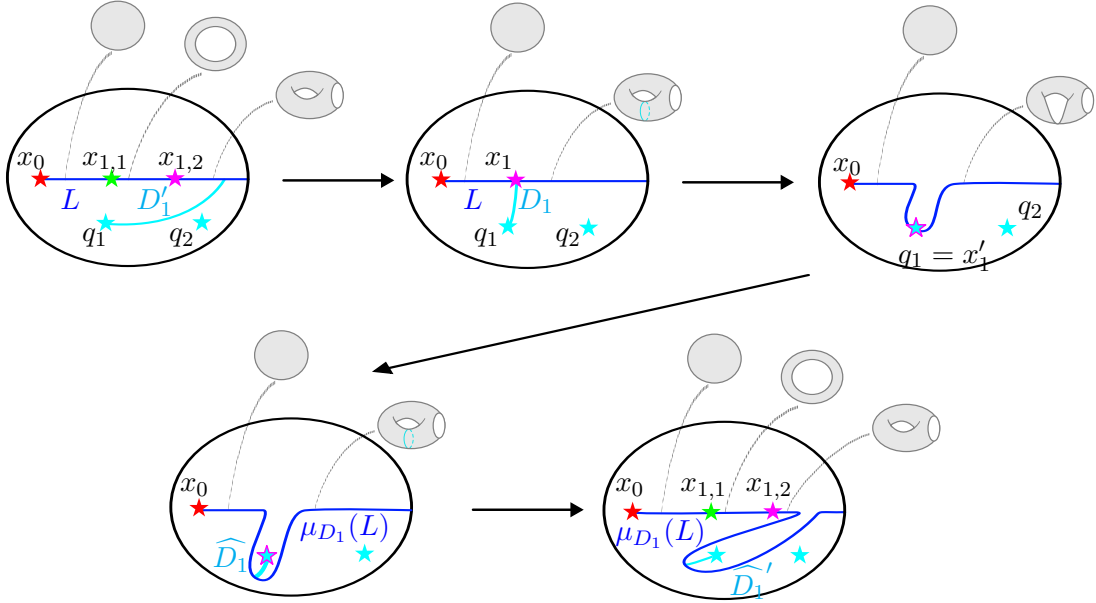
\begin{figure}[ht]
\begin{overpic}[scale=0.65]{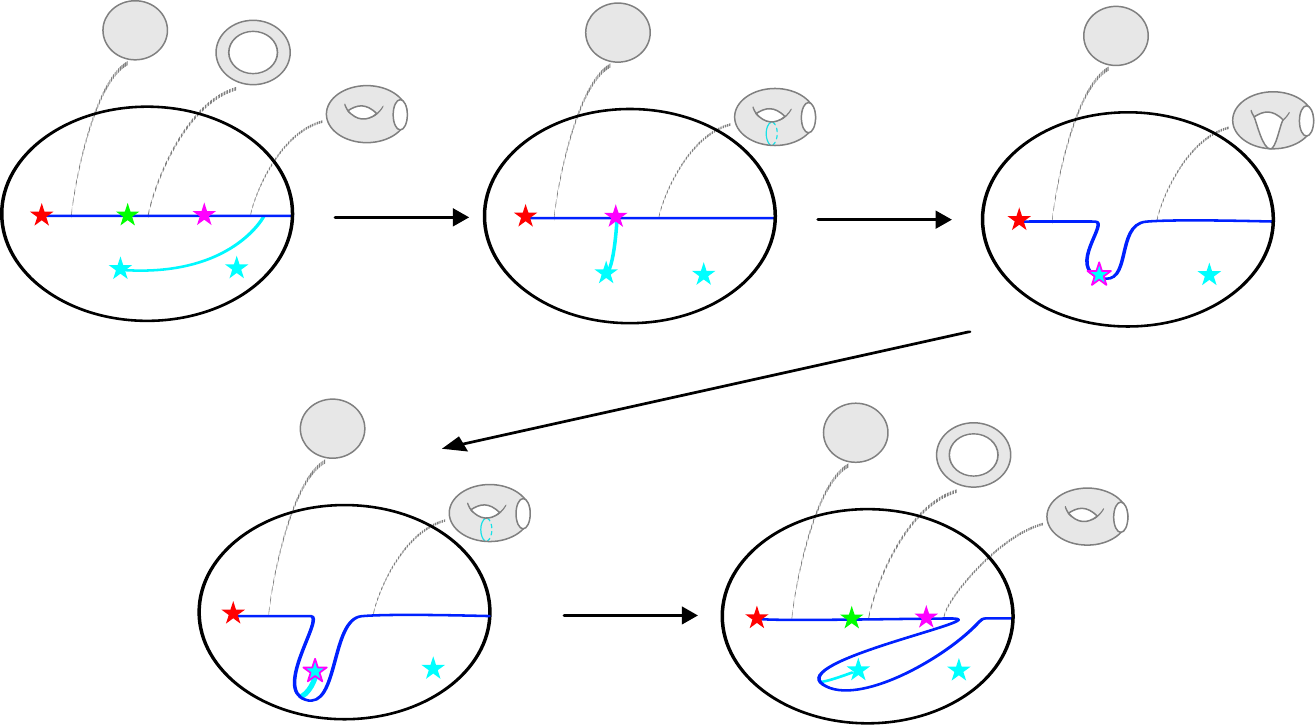}
\put(1,40){\small{$x_0$}}
\put(7,40){\small{$x_{1,1}$}}
\put(14,40){\small{$x_{1,2}$}}
\put(5,36){\small\textcolor{blue}{$L$}}
\put(12,36){\small\textcolor{cyan}{$D_1'$}}
\put(7,33){\small{$q_1$}}
\put(15,33){\small{$q_2$}}

\put(38,40){\small{$x_0$}}
\put(44,40){\small{$x_{1}$}}
\put(42,36){\small\textcolor{blue}{$L$}}
\put(47,36){\small\textcolor{cyan}{$D_1$}}
\put(43,33){\small{$q_1$}}
\put(50.5,33){\small{$q_2$}}

\put(76,39.5){\small{$x_0$}}
\put(82,31){\small{$q_1 = x_1'$}}
\put(93,35.5){\small{$q_2$}}

\put(16,10){\small{$x_0$}}
\put(27,6){\small\textcolor{blue}{$\mu_{D_1}(L)$}}
\put(19,3){\small\textcolor{cyan}{$\widehat{D_1}$}}

\put(56,10){\small{$x_0$}}
\put(62,10){\small{$x_{1,1}$}}
\put(68,10){\small{$x_{1,2}$}}
\put(56,5.5){\small\textcolor{blue}{$\mu_{D_1}(L)$}}
\put(66.5,0.5){\small\textcolor{cyan}{$\widehat{D_1}'$}}
\end{overpic}
    \caption{The steps in realizing the mutated Lagrangian filling as an arc-admissible filling. The first picture is the standard Weinstein Lefschetz picture while the second is the degenerate standard Weinstein Lefschetz picture.}
    \label{fig: mutation}
\end{figure}

\begin{proposition}\label{prop: once_mutated_filling}
    Consider the Weinstein pair $(B^4, \Lambda)$ and a regular Lagrangian filling $L$ of $\Lambda$ with an efficient Morse function $f$. Suppose there is an associated CAL-skeleton $(L, \mathcal{D}, \Gamma)$.  Let $p_L$ be an admissible Lefschetz fibration for the triple $(B^4,L,f)$ which gives the standard Weinstein Lefschetz picture. Then, any Lagrangian $L'$ obtained by mutations on $L$ along disks in $\mathcal{D}$ is arc-admissible for $p_L$. 
\end{proposition}

\begin{proof}   
    Let $L' = \mu_{D}(L_1)$, where $D \in \mathcal{D}$ and $L_1$ is obtained from mutations on $L$ along disks in $\mathcal{D}_{L_1} \subseteq \mathcal{D}$, and $\partial D = \gamma_1 \subset L_1$. The case $\mathcal{D}_{L_1} = \emptyset$ is the base case, where $L_1 = L$. We proceed inductively, by showing that each mutated Lagrangian $L'$ is arc-admissible for $p_L$. In fact, we show that $p_L(L')$ is obtained by pushing a part of the arc $p_L(L)$ to the right of all of $\{x_{1, i}\}$, over the critical values corresponding to the disks in $\mathcal{D}_{L_1}$ along arcs.
 
    By the induction hypothesis, $p_L(L_1)$ is an arc obtained by pushing the rightmost segment of $p_L(L)$ past the critical values of $p_L|_{W_L}$ along arcs.
    
    Let $\alpha \subset \D^2$ be an arc between the critical value $q_1$ of $W_L$ and an arbitrary regular value $x \in p_L(L_1) \subset \D^2$, chosen so that the interior of $\alpha$ avoids $p_L(L_1)$. Associated to $\alpha$ is a vanishing cycle $c_1 \subset p_L^{-1}(x)$, consisting of the points in the regular fiber over $x$ which parallel transport into the critical point corresponding to $q_1$. Let $D'$ denote the Lagrangian vanishing thimble over $\alpha$ with boundary $c_1$.
    
    A priori, the vanishing cycle $c_1 \subset F$ depends on the choice of $\alpha$. Nevertheless, we claim that $D'$ can be extended 
    (up to Hamiltonian isotopy) into the $\mathbb{L}$-compressing disk $D$ for $L_1$ with boundary $\gamma_1$. The corresponding coupled Weinstein Lefschetz pictures are described in \cref{fig: mutation}. Further, we claim that after a homotopy of $p_L$, this embedded curve $\gamma_1$ can be seen on a (singular) fiber of the Lefschetz fibration.
    
    By a homotopy of the coupled Weinstein Lefschetz handlebody, we can modify it to the degenerate standard Weinstein Lefschetz picture. This means $p_L^{-1}(x_1) \cap L = p_L^{-1}(x_1)\cap L_1$ is obtained by:
    \begin{itemize}
    \item starting with the boundary of the $0$-handle of $L$,
    \item attaching the stable manifolds of each index $1$ point, and
    \item contracting the boundary of the $0$-handle along these stable manifolds.
    \end{itemize}
    In particular, given any element in the basis of $H_1(L)$ corresponding to the 1-skeleton of $L$ given by $f$, there is a representative embedded in $p_L^{-1}(x_1) \cap L$. Now, note that the Morse function $f$ on $L$ gives a Morse function on $L_1$, which is still given by $p_L|_{L_1}$. So, we can talk about the compatibility of the CAL-skeleton $(L_1, \mathcal{D}, \mu(\Gamma))$ with $p_L$. If the CAL-skeleton is compatible, then $c_1$ can be isotoped to be embedded in $L_1 \cap p_L^{-1}(x_1)$. If not, note that $\gamma_1$ is embedded in $L_1$, since $D$ is an $\mathbb{L}$-compressing disk for $L_1$. Given any basis of $H_1(L_1)$, an isotopic copy of an embedded curve on $L_1$ can be described as a linear combination of curves in that basis. All the curves in the 1-skeleton of $L$ corresponding to $f$, forming a basis for $H_1(L) = H_1(L_1)$, are embedded in $L_1 \cap p_L^{-1}(x_1)$. By handleslides of the 1-handles of $L$, which correspond to the 1-handles of $L_1$, we can ensure that $D' \cap p_L^{-1}(x_1)$ can be isotoped to be embedded in $L_1 \cap p_L^{-1}(x_1)$. In other words, there exists a homotopy of the handle structure on $L_1$, and a corresponding Weinstein Lefschetz homotopy of $p_L$, so that we obtain $D$ as arc-admissible, i.e. its boundary $\gamma_1 \subset L_1$ is also an embedded curve in a (singular) fiber of the Lefschetz fibration, as shown in the second image of \cref{fig: mutation}. On the nodal fiber $p_L^{-1}(x_1)$, $\partial D$ is a simple closed curve that is isotopic to $c_1$. 
    
    Now, we give the Weinstein Lefschetz description of mutation associated of $L_1$ along $D$. Push the arc $p_L(L_1)$ along $\alpha$ so that it intersects $q_1$ at the critical value $x_1$, now denoted $x_1'$. The disk $D$ contracts to a point, as shown in the third image of \cref{fig: mutation}. We obtain a family of exact Lagrangians $\{L^t\}_{t \in [0,1]}$ where $L^0 = L_1$, $L^t$ is Hamiltonian isotopic to $L_1$ for $t \in [0,1)$, and $L^1$ is an immersed Lagrangian obtained by contracting $\gamma_1 \subset L_1$ to a point. If we push the Lagrangian $L_1$ further along the arc past $p_1$, its lift is an embedded Lagrangian which is $\mu_{D}(L_1) = L'$, the Lagrangian mutation of $L_1$ along $D$, which is smoothly isotopic to $L_1$. The curve $\gamma_1 \subset \mu_{D}(L_1)$ bounds a Lagrangian disk $\widehat{D}$. This follows because,  by the argument above, $L_1 \cup D$ forms a mutation configuration in the sense of \cite[Section 4.5]{pascaleff-tonkonog} or \cite[Section 2]{yau_surgery}. Pushing the arc $p_L(L_1)$ past $q_1$ corresponds exactly to \cite[Figure 2]{pascaleff-tonkonog}, which characterizes the local picture of a Lagrangian mutation under a Lefschetz map. By a coupled Weinstein Lefschetz homotopy, we can isolate the critical values $\{x_{1,i}\}$ along $p(\mu_{D_1}(L))$; this corresponds to the last image of \cref{fig: mutation}, which depicts an ambient Weinstein structure identical to the initial one.

    For this whole procedure, we go from the first image of \cref{fig: mutation} to the last image while keeping the Weinstein Lefschetz structure fixed, only performing exact Lagrangian isotopies and Lagrangian mutations. This disrupts the coupled structure in the intermediate steps, but the above argument that respects the coupled structure shows that the mutated Lagrangian is in fact arc-admissible for the same Weinstein Lefschetz structure, and the arc is described as in the statement of the lemma. This completes the proof.
\end{proof}

Given a Weinstein pair $(B^4,L)$ as above, we can encode the data of the CAL-skeleton $(L, \mathcal{D}, \Gamma)$ schematically in an admissible Lefschetz fibration $p_L$ as, for example, in the first image of \cref{fig: trefoil_fillings}, described later in \cref{example:trefoil} for fillings of the trefoil: an arc representing $L$ and arcs representing the disks $D_i$ joining points on $p(L)$ with critical points of $W_L$. As the proof shows, the latter arcs do not lift directly to $\mathbb{L}$-compressing disks, but there is no loss of generality in associating disks to the arcs.

\begin{figure}[ht]
	\begin{overpic}[scale=0.8]{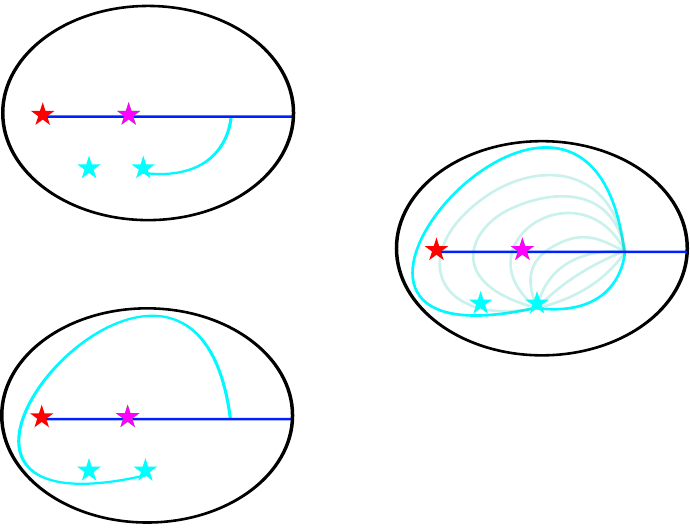}
\put(4,62){\small{$x_0$}}
\put(17,62){\small{$x_{1}$}}
\put(44,58){\small\textcolor{blue}{$L$}}
\put(33,54){\small\textcolor{cyan}{$D_{\alpha}$}}

\put(34,19){\small\textcolor{cyan}{$D_{\alpha'}$}}

\put(70,47){\small\textcolor{cyan}{$\{D_{\alpha_t}\}$}}

    \end{overpic}
    \caption{The situation described in \cref{lem: arc-disk}, with different arcs giving isotopic $\mathbb{L}$-compressing disks. The isotopy $\{D_{\alpha_t}\}_{t\in[0,1]}$ between $D_{\alpha_0}:=D_{\alpha}$ and $D_{\alpha_1}:=D_{\alpha'}$ is in the rightmost figure. By assumption, the curve $c_{\alpha} \subset F$ is disjoint from all the vanishing cycles corresponding to the other critical values.}
    \label{fig: arc_isotopy}
    \end{figure}
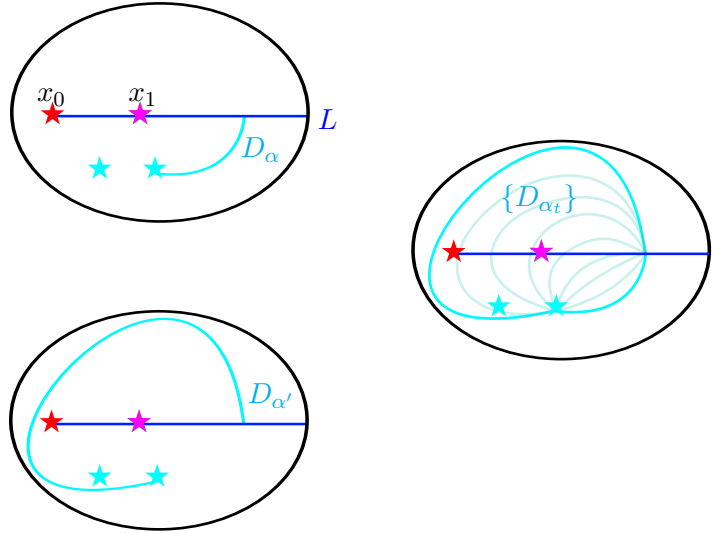

We now give a sufficient condition for two arcs with end-points at $x_1$ and a critical value of $p_L$ restricted to $W_L$, to lift to Hamiltonian isotopic (relative boundary) $\mathbb{L}$-compressing disks for an arc-admissible Lagrangian for $p_L$.

\begin{lemma}\label{lem: arc-disk}
    Fix a degenerate standard Weinstein Lefschetz picture as in \cref{def: standard_WLp}. Let $L_1$ be an arc-admissible Lagrangian for $p_L$ obtained by mutations on $L$. If the vanishing cycle corresponding to $q_j$ is disjoint on a regular fiber with the vanishing cycles for $p_L$ corresponding to $x_0$ and $q_i$ for $\{i \neq j\}$, then (after a homotopy of $p_L$) any arc with end-points at $x_1$ and $q_j$ lifts to a well-defined $\mathbb{L}$-compressing disk for $L_1$.
\end{lemma}

\begin{proof}
    The Hamiltonian isotopy relative boundary is described in \cref{fig: arc_isotopy}.
\end{proof}

It follows from \cref{lem: arc-disk} that in some cases there can be infinitely many homotopy classes of arcs representing the once mutated filling $\mu_{D_1}(L)$, even after fixing $p(L)$. A concrete example is given below in \cref{example:trefoil}, describing the fillings of the trefoil $\Lambda = \Lambda(2,3)$.

\subsection{CAL-skeleta associated to fillings}

In this subsection we discuss how to find or construct fillings that have associated CAL-skeleta, and show that \cref{thm: mutation} applies to all conjectured fillings of Legendrian links in $(S^3, \xi_{\mathrm{st}})$. Recall from the introduction that if $\beta$ is a positive braid word, we let $\Lambda_{\beta}$ denote its rainbow closure as in \cref{fig:rainbow}.

\begin{example}\label{eg: decomposable}
    A decomposable filling of $\Lambda_\beta$ can be built by Legendrian isotopies, births, and pinch moves in the front projection as in \cref{fig: decomposable}. Different sequences of pinches will typically give distinct (up to Hamiltonian isotopy) fillings. 
\end{example}

\begin{figure}[ht]
	\begin{overpic}[scale=.7]{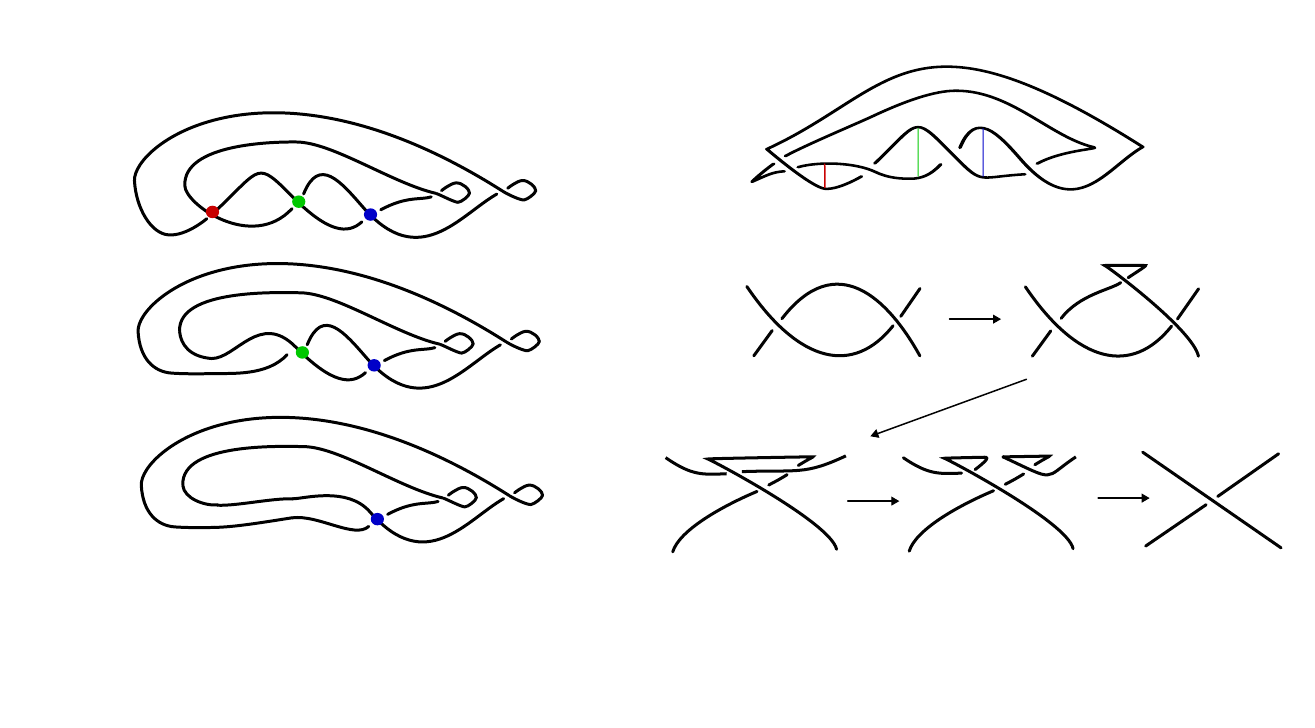} 
    \end{overpic}
    \caption{The initial decomposable filling $L_{\mathrm{init}}$ of the max-tb trefoil $\Lambda(2,3)$, described as cobordisms of the Lagrangian projection. The same knot is drawn in the front projection on the right, so the Reeb chords are visible. The saddle move in the Lagrangian projection is described as a sequence of Legendrian isotopies and a single pinch in the front projection on the right, following \cite{hughes_a-type}.}
    \label{fig: decomposable}
    \end{figure}

The initial filling $L_{\mathrm{init}}$ used by \cite{casals_gao24} is Hamiltonian isotopic to a decomposable filling built using a single birth, Legendrian isotopies, and pinch moves in the front projection. The algorithm in \cite{ConwayEtnyreTosun2021Disks} that establishes regularity of $L_{\mathrm{init}}$ actually shows that $L_{\mathrm{init}}$ is part of a CAL-skeleton for the pair $(B^4, \Lambda_\beta)$. We make this precise in the following, noting that (a variation of) the result is known to experts, e.g. \cite[Remark 4.4]{casals_gao24}. This is a restatement of \cref{prop:decomposable-skel}.

\begin{proposition}\label{prop: decomposable_to_skeleton}
    Given a connected decomposable filling $L$ of a Legendrian link $\Lambda_\beta \subset (S^3, \xi_{\mathrm{st}})$, the standard Weinstein structure $\mathcal{W}$ on the Weinstein pair $(B^4,\Lambda_\beta)$ can be homotoped so that it has an arboreal skeleton consisting of $L$ and disks $D_i$ attached to $L$ cleanly along curves $\gamma_i$, such that the collection $\{\gamma_i\}$ forms a basis of $H_1(L)$.
\end{proposition}

\begin{proof}
    The proof follows from the Weinstein structure associated to a decomposable cobordism described by \cite[Section 3.2]{ConwayEtnyreTosun2021Disks}. The coupled Weinstein structure on $(B^4, L)$ (\cref{def:coupled_handle_decomposition}) can be built using a single coupled 0-handle $(h^4_0, h^2_0)$, and pairs of coupled Weinstein $1$-handles $(h^4_1,h^2_1)$ with canceling Weinstein $2$-handles $(h^4_2,\emptyset)$ (c.f. \cite[Figure 7]{ConwayEtnyreTosun2021Disks}). The attaching curves of the $2$-handles are the boundaries of the core disks, which are parallel to the core of the $1$-handles. By a $0$-Weinstein homotopy, the Lagrangian disks can be extended into the $1$-handles so that their boundaries lie on the Lagrangian handles inside the coupled handles. This produces the arboreal skeleton given by the union of $L$ and the disks $D_i$. Since a basis of $H_1(L)$ is given by the cores corresponding to each of the $1$-handles, and these are the boundaries of the disks $D_i$, the lemma follows. 
\end{proof}

\begin{remark}
    The construction of the associated CAL-skeleton above works for any decomposable filling of a Legendrian link in $(S^3,\xi_{\mathrm{st}})$.  However, if $\Lambda$ is not the rainbow closure of a positive braid, we may not get a basis of $H_1(L)$ from the $\mathbb{L}$-compressing system. For example, the decomposable fillings of the knot $m(9_{46})$ are disks, hence have vanishing first homology.
\end{remark}

\begin{figure}[ht]
	\begin{overpic}[scale=.5]{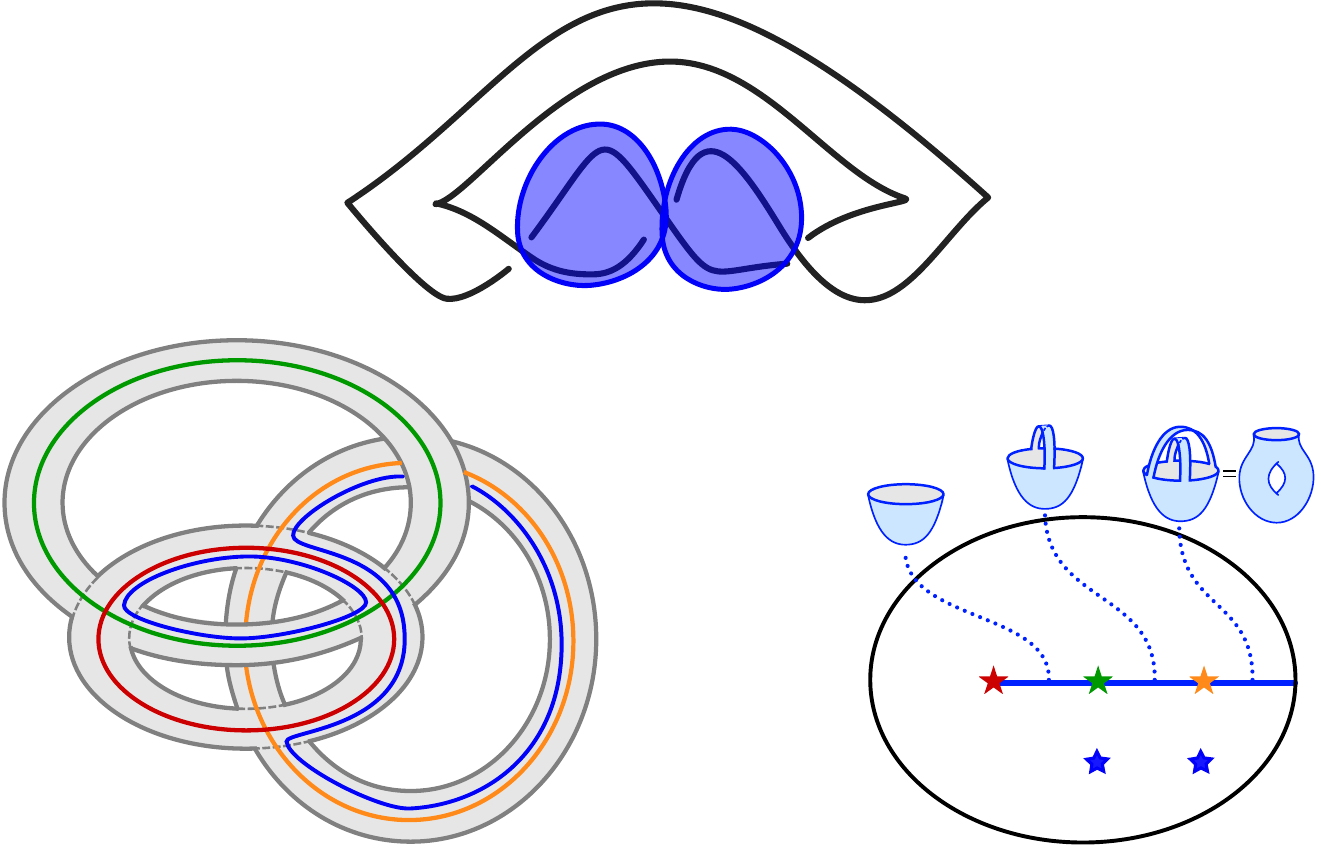} 
    \end{overpic}
    \caption{Counter-clockwise from top: The $\mathbb{L}$-compressing system on the initial filling $L
    _{\mathrm{init}}$ for the max-tb trefoil, the page of an admissible Lefschetz fibration for the pair $(B^4, L_{\mathrm{init}})$, and the standard Weinstein Lefschetz picture for the same.}
    \label{fig: trefoil_LF}
    \end{figure}

 \begin{figure}[ht]
	\begin{overpic}[scale=.7]{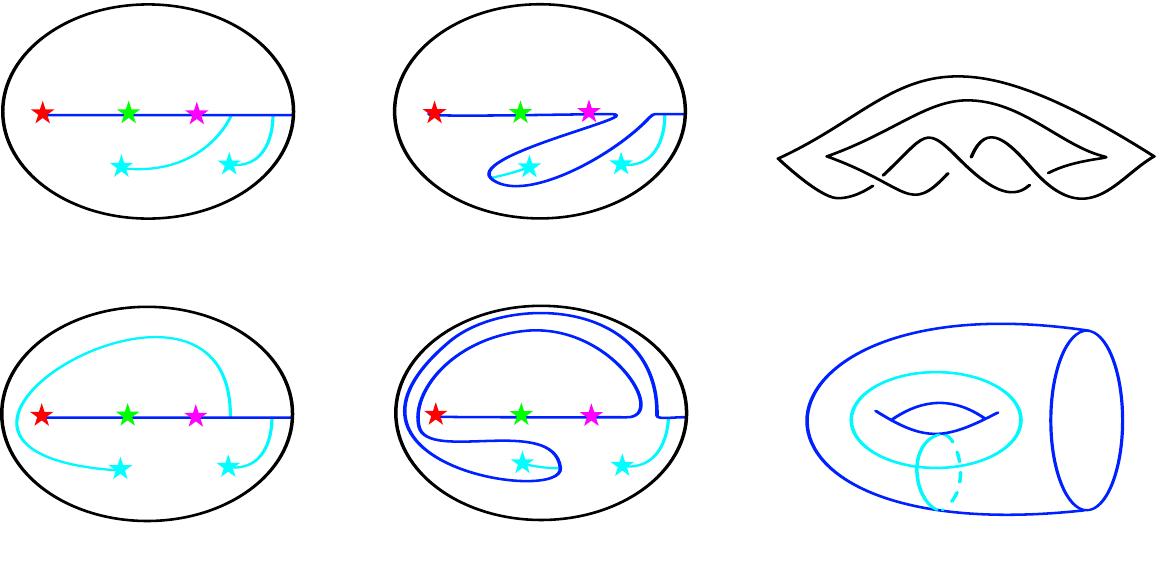}
    \put(8,28){\small\textcolor{blue}{$L_{\text{init}}$}}
    \put(40,28){\small\textcolor{blue}{$\mu_{D_1}(L_{\text{init}})$}}
    \put(8,3){\small\textcolor{blue}{$L_{\text{init}}$}}
    \put(40,3){\small\textcolor{blue}{$\mu_{D_1}(L_{\text{init}})$}}

\put(2,42){\small{$x_0$}}
\put(8,42){\small{$x_{1,1}$}}
\put(15,42){\small{$x_{1,2}$}}
\put(12,37){\small\textcolor{cyan}{$D_1$}}
\put(19,37){\small\textcolor{cyan}{$D_2$}}
\put(7,34){\small{$q_1$}}
\put(16,34){\small{$q_2$}}

\put(37,42){\small{$x_0$}}
\put(43,42){\small{$x_{1,1}$}}
\put(49,42){\small{$x_{1,2}$}}
\put(39,37){\small\textcolor{cyan}{$D_1'$}}
\put(50.5,34){\small\textcolor{cyan}{$D_2$}}

\put(82,30){\small{$T_{2,3}$}}

\put(3,15.5){\small{$x_0$}}
\put(8,15.5){\small{$x_{1,1}$}}
\put(15,15.5){\small{$x_{1,2}$}}
\put(12,19){\small\textcolor{cyan}{$D_1$}}
\put(19,11){\small\textcolor{cyan}{$D_2$}}

\put(37,15.5){\small{$x_0$}}
\put(43,15.5){\small{$x_{1,1}$}}
\put(49,15.5){\small{$x_{1,2}$}}
\put(40.5,10){\small\textcolor{cyan}{$D_1'$}}
\put(52.5,11){\small\textcolor{cyan}{$D_2$}}

\put(77,8){\small\textcolor{cyan}{$\gamma_1$}}
\put(79,19){\small\textcolor{cyan}{$\gamma_2$}}
    \end{overpic}
    \caption{On the top left, an admissible Lefschetz fibration for $(B^4, L_{\mathrm{init}})$, showing the arboreal skeleton given by $L_{\mathrm{init}}\cup D_1 \cup D_2$. The boundaries of the $\mathbb{L}$-compressing disks on $L_{\mathrm{init}}$ are described in the bottom right. The leftmost picture on the bottom is another picture of the same skeleton, where a different arc is used to represent $D_1$. The two middle pictures are two different representations of the filling given by mutation along $D_1$.}
    \label{fig: trefoil_fillings}
    \end{figure}

\begin{example}\label{example:trefoil}
    Consider the initial filling $L_{\mathrm{init}}$ of the trefoil from \cref{fig: decomposable}.
    The associated $\mathbb{L}$-compressing system and the standard Weinstein Lefschetz picture
    of the pair $(B^4, L_{\mathrm{init}})$ are described in \cref{fig: trefoil_LF}. Different sequences of mutations on $L_{\mathrm{init}}$ yield $5$ different fillings of $\Lambda(2,3)$ \cite{STWZ, TreumannZaslow,ekholm2012exactcobordisms, hughes_a-type}, all of which are described in \cref{fig: same_arc_different}. From \cref{fig: arc_isotopy} and \cref{fig: trefoil_fillings}, we see that in the Lefschetz fibration corresponding to the standard Weinstein Lefschetz picture of $L_{\mathrm{init}}$, there are infinitely many distinct homotopy classes of arcs that are the image of fillings Hamiltonian isotopic to $\mu_{D_1}(L_{\mathrm{init}})$. On the other hand, in \cref{fig: same_arc_different} we observe that $\mu_{D_2}\mu_{D_1}(L_{\mathrm{init}})$ and $\mu_{D_1}\mu_{D_2}(L_{\mathrm{init}})$ have representatives up to Hamiltonian isotopy with the same image in the base of the Lefschetz fibration. However, if we include the full data of the curves on the once-punctured torus specifying the associated $\mathbb{L}$-compressing system, we see that the skeleta for the two rightmost pictures of \cref{fig: same_arc_different} are distinct. Let $(L_{\mathrm{init}}, (D_1,D_2), ((0,1),(1,0)))$ denote the initial CAL-skeleton as in the bottom right of \cref{fig: trefoil_fillings}. Then the CAL-skeleton associated to $\mu_{D_1D_2}(L_{\mathrm{init}})$ is 
    \[
    (\mu_{D_1D_2}(L_{\mathrm{init}}),\,\,  (\widehat{D_1}, \widehat{D_2}),\,\, ((1,1),(2,1))),
    \]
    and the CAL-skeleton associated to $\mu_{D_2D_1}(L_{\mathrm{init}})$ is 
    \[
    (\mu_{D_2D_1}(L_{\mathrm{init}}),\,\, (\widehat{D_1}, \widehat{D_2}),\,\,((1,1),(0,1))).
    \]
    These two fillings were distinguished using augmentations in \cite{ekholm2012exactcobordisms}.
\end{example}

\begin{figure}[ht]
	\begin{overpic}[scale=.8]{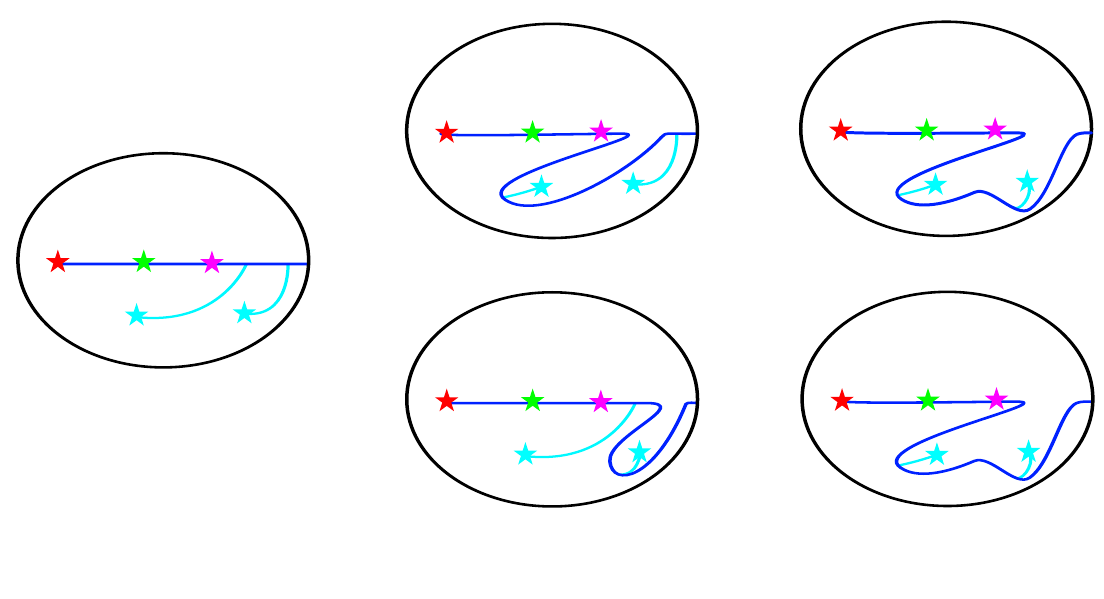} 
    \put(10,18){\small\textcolor{blue}{$L_{\text{init}}(L)$}}
    \put(44,29.5){\small\textcolor{blue}{$\mu_{D_1}(L_{\text{init}})$}}
    \put(79,29.5){\small\textcolor{blue}{$\mu_{D_2}\mu_{D_1}(L_{\text{init}})$}}
    \put(44,5){\small\textcolor{blue}{$\mu_{D_2}(L_{\text{init}})$}}
    \put(79,5){\small\textcolor{blue}{$\mu_{D_1}\mu_{D_2}(L_{\text{init}})$}}

\put(3,32){\small{$x_0$}}
\put(9,32){\small{$x_{1,1}$}}
\put(16,32){\small{$x_{1,2}$}}
\put(16,27){\small\textcolor{cyan}{$D_1$}}
\put(21.5,27){\small\textcolor{cyan}{$D_2$}}
\put(9,24){\small{$p_1$}}
\put(19,24){\small{$p_2$}}

\put(39,44){\small{$x_0$}}
\put(46,44){\small{$x_{1,1}$}}
\put(52,44){\small{$x_{1,2}$}}
\put(42,35){\small\textcolor{cyan}{$D_1'$}}
\put(53,35){\small\textcolor{cyan}{$D_2$}}

\put(75,44){\small{$x_0$}}
\put(82,44){\small{$x_{1,1}$}}
\put(88,44){\small{$x_{1,2}$}}
\put(78,35){\small\textcolor{cyan}{$D_1'$}}
\put(89,37){\small\textcolor{cyan}{$D_2'$}}

\put(39,20){\small{$x_0$}}
\put(46,20){\small{$x_{1,1}$}}
\put(52,20){\small{$x_{1,2}$}}
\put(51,14){\small\textcolor{cyan}{$D_1$}}
\put(52,10){\small\textcolor{cyan}{$D_2'$}}

\put(75,20){\small{$x_0$}}
\put(82,20){\small{$x_{1,1}$}}
\put(88,20){\small{$x_{1,2}$}}
\put(78,11){\small\textcolor{cyan}{$D_1'$}}
\put(89,13){\small\textcolor{cyan}{$D_2'$}}
    \end{overpic}
    \caption{On the top left, an admissible Lefschetz fibration for $(B^4, L_{\mathrm{init}})$, showing the arboreal skeleton given by $L_{\mathrm{init}}\cup D_1 \cup D_2$. The rightmost picture in each row are apparently the same, however the corresponding fillings of $\Lambda(2,3)$ are a priori different as they come from different sequences of mutations and can be distinguished using invariants.}
    \label{fig: same_arc_different}
\end{figure}

For $\Lambda = \Lambda_{\beta}$, Casals and Gao \cite{casals_gao24} construct an initial filling $L_{\mathrm{init}}$ as a conjugate surface, such that its associated $\mathbb{L}$-compressible system has an intersection quiver $Q_\beta$ (Definition 3.4 loc. cit.). They describe how the $\mathbb{L}$-compressible system changes during disk surgeries on $L_{\mathrm{init}}$, and show that after an arbitrary number of mutations, the boundaries of the disks can always be represented by embedded curves. The results of \cite{casals-li22} show that this conjugate filling is Hamiltonian isotopic to a decomposable filling built as in \cref{eg: decomposable}. Hence by \cref{prop: decomposable_to_skeleton}, there is a Weinstein homotopy of the standard $B^4$ such that $L_{\mathrm{init}}$ is part of an arboreal skeleton. Moreover, the $\mathbb{L}$-compressible system used in \cite{casals_gao24} becomes the system that appears in \cref{prop: decomposable_to_skeleton}, bounding the cores of $W_{L_{\mathrm{init}}}$.

In general, following the discussion in \cite[Section 4]{casals_gao24},  the mutated CAL-skeleton $(\mu_{D_i}(L), \{D_i'\}, \{\gamma_i'\})$ associated to $\mu_{D_1}(L)$ is the 3-tuple $ (\mu_{D_1}(L), \{\widehat{D_1}, D_2, \dots\}, \{\gamma_i'\})$ where $\gamma_1' = \gamma_1$, $\gamma_i' = \gamma_i + k\gamma_1$ if $\gamma_1$ and $\gamma_i$ intersect positively $k$ times (see Figures 16 and 17 in \cite{casals_gao24}; in their terminology this change is called a $\gamma_1$-exchange). A priori, it is not clear in general that an $\mathbb{L}$-compressing system after mutation gives an $\mathbb{L}$-compressing system. However, for positive rainbow closures, we can leverage the results of \cite{casals_gao24}.

\begin{proof}[Proof of \cref{cor:CG-fillings}]
    It is shown in \cite{casals_gao24} that for all seeds in the cluster algebra associated to $\Lambda_\beta$ there are exact fillings such that the $\mathbb{L}$-compressing systems are represented by embedded simple closed curves, and they stay as such after any sequence of mutations along $\mathcal{D}$. Let $p_\beta$ be the Lefschetz fibration corresponding to the standard Weinstein Lefschetz picture of $L_{\mathrm{init}}$, coming from \cref{thm:main}. Applying \cref{prop: once_mutated_filling}, it follows that for every seed of the cluster variety $X_\beta$ there is a corresponding filling $L$ of $\Lambda_\beta$ such that $p_{\beta}(L)$ is an arc. 
\end{proof}

\begin{remark}
    Decomposable fillings can be constructed for larger classes of links in $(S^3, \xi_{\mathrm{st}})$, such as $(-1)$-closures of positive braids and twist knots. There exist cluster structures on the sheaf moduli of certain $(-1)$-closures of positive braids, but it is not known yet how to construct a filling for every seed (see \cite[Remark 3.12]{casals_gao24}). 
\end{remark}

\begin{remark}
    We can generalize \cref{example:trefoil} to all $(2,n)$ torus links, i.e. rainbow closures of positive $2$-braids, to obtain a surjective-but-not injective (either way) correspondence between fillings of $\Lambda(2,n)$ and arcs on the plane (seen as the codomain of a standard Weinstein Lefschetz fibration for the initial filling $L_{\mathrm{init}}$) passing through $x_0$ and the $x_{1,i}$'s. In this setting, analogous to \cref{fig: trefoil_LF}, the vanishing cycles obtained from \cref{prop:decomposable-skel} corresponding to the $q_i$ critical values are all disjoint in a regular fiber. Thus, by \cref{prop: once_mutated_filling} and \cref{lem: arc-disk}, for fixed $L_{\mathrm{init}}$ and $p_{L_{\mathrm{init}}}$, all possible arcs on the punctured $D^2$ are images under $p_{L_{\mathrm{init}}}$ of exact fillings of $\Lambda(2,n)$, obtained from $L_{\mathrm{init}}$ via mutation. Further, the Hamiltonian isotopy class of a filling corresponding to an arc $\alpha$ can be read off of a sequence of arcs joining $p_{L_{\mathrm{init}}}(L_{\mathrm{init}})$ to the critical values $q_i$, along which the arc $p_{L_{\mathrm{init}}}(L_{\mathrm{init}})$ can be pushed over the critical values to obtain $\alpha$. There need not be a unique such sequence, which accounts for the non-injectivity from arcs to fillings.
\end{remark}

\bibliography{references}
\bibliographystyle{amsalpha}

\end{document}